\documentclass[a4paper,10pt]{amsart}
\usepackage[utf8]{inputenc}
\usepackage[T1]{fontenc}

\usepackage{amsmath}
\usepackage{amssymb}
\usepackage{amsthm}
\usepackage{mathtools}
\usepackage{bm}
\usepackage{cancel}
\usepackage{eucal} 
\usepackage{mathdots}
\usepackage[colorlinks=true,linkcolor=red,citecolor=blue,urlcolor=green]{hyperref}
\usepackage[capitalise]{cleveref}
\usepackage{enumitem}
\usepackage{mathrsfs}  
\usepackage[normalem]{ulem}
\usepackage{tikz-cd}
\usepackage{adjustbox}
\usepackage{contour}
\usepackage{ulem}
\usepackage{faktor}

\numberwithin{equation}{section}
\crefname{equation}{}{}

\crefformat{section}{\S#2#1#3}
\crefmultiformat{section}{\S\S#2#1#3}{ and~#2#1#3}{, #2#1#3}{, and~#2#1#3}

\contourlength{0.8pt}

\newtheorem{theorem}{Theorem}[section]
\newtheorem{prop}[theorem]{Proposition}
\newtheorem{cor}[theorem]{Corollary}
\newtheorem{lemma}[theorem]{Lemma}

\theoremstyle{definition}
\newtheorem{dfn}[theorem]{Definition}
\newtheorem{rmk}[theorem]{Remark}

\newtheorem{ex}[theorem]{Example}

\DeclareMathOperator{\Hom}{\mathsf{Hom}}

\DeclareMathOperator{\Ext}{\mathsf{Ext}}
\DeclareMathOperator{\Tor}{\mathsf{Tor}}

\newcommand{\Acal}{\mathcal{A}}
\newcommand{\Bcal}{\mathcal{B}}

\newcommand{\Ical}{\mathcal{I}}

\newcommand{\Xcal}{\mathcal{X}}

\newcommand{\op}{\mathsf{op}}

\newcommand{\Mod}[1]{\mathsf{Mod}\mbox{-}#1}
\newcommand{\lMod}[1]{#1\mbox{-}\mathsf{Mod}}
\newcommand{\rMod}[1]{\mathsf{Mod}\mbox{-}#1}

\newcommand{{\tst}}{\textit{t}-}
\newcommand{\pd}{\mathsf{pd}}
\newcommand{\fd}{\mathsf{fd}}
\newcommand{\id}{\mathsf{id}}

\newcommand{\sfindim}{\mathsf{sfindim}} 

\newcommand{\wD}{\mathsf{w.D}}
\newcommand{\gD}{\mathsf{gl.D}}

\def\nid{\mathop{\rm \mathsf{FP}_{n}\text{-}id}\nolimits}
\def\oneid{\mathop{\rm \mathsf{FP}_{1}\text{-}id}\nolimits}
\def\nfd{\mathop{\rm \mathsf{FP}_{n}\text{-}fd}\nolimits}

\def\glD{\mathop{\rm glD}\nolimits}

\def\gid{\mathop{\rm \mathsf{FP}_n\text{-}ID}\nolimits}
\def\gfd{\mathop{\rm \mathsf{FP}_n\text{-}FD}\nolimits}
\def\gidi{\mathop{\rm \mathsf{FP}_{\infty}\text{-}ID}\nolimits}
\def\gidone{\mathop{\rm \mathsf{FP}_{1}\text{-}ID}\nolimits}
\def\gfdi{\mathop{\rm \mathsf{FP}_{\infty}\text{-}FD}\nolimits}

\newcommand{\newterm}[1]{\textit{#1}}

\title[Regularity and $\mathsf{K}_0$-Regularity under Finiteness Conditions.]
{Regularity and $\mathsf{K}_0$-Regularity under Finiteness Conditions.}

\author{Rafael Parra}
\address[R. Parra]{IMERL, Facultad de Ingenier\'\i a, Universidad de la Rep\'ublica \\
Julio Herrera y Reissig 565, 11.300, Montevideo, Uruguay}
\email{rparra@fing.edu.uy}


\thanks{}

\begin{document}

\begin{abstract}
The purpose of this work is to investigate various notions of regularity from the perspective of finiteness conditions, with the ultimate goal of identifying broad classes of rings that are $\mathsf{K}_0$-regular. 
In this direction, we revisit the classical concepts of coherence and von Neumann regularity, and establish new characterizations.  We then focus on the study of \newterm{$n$-coherent regular rings}, recently introduced in \cite{ep}, and analyze their $\mathsf{K}$-theoretic behavior. Finally, we present applications illustrating how these approaches provide examples of $\mathsf{K}_0$-regular rings.
\end{abstract}

\maketitle
\tableofcontents
\section{Introduction}  
Algebraic $\mathsf{K}$-theory associates to a ring $R$ a sequence of groups $\{\mathsf{K}_i(R)\}_{i \in \mathbb{Z}}$. Computing these groups is extremely difficult and only possible in special cases. Given this complexity, modern approaches often focus on understanding how the $\mathsf{K}$-groups behave under equivalences between related categories, allowing the transfer of information from one ring to another. This perspective not only offers practical methods of computation but also contributes to a broader understanding of algebraic $\mathsf{K}$-groups. 

 Let $R$ be an associative ring with unity. An $R$-module $M$ is said to be finitely $n$-presented if there exists an exact sequence
$$F_n \to F_{n-1} \to \cdots \to F_1 \to F_0 \to M \to 0,$$
where each $F_i$ is a finitely generated projective (or free) $R$-module for $0 \leq i \leq n$. When $n=1$, this reduces to the classical notion of a finitely presented module, and for $n=0$, it corresponds to a finitely generated module. It is a well-known result that, for left coherent regular rings, 
$$\mathsf{K}_i(R) \cong \mathsf{K}_i(\mathsf{FP}_1(R)), \quad i \ge 0,$$
where $\mathsf{FP}_1(R)$ denotes the class of finitely presented $R$-modules. Motivated by the principle that many homological results extend by replacing finitely presented modules with finitely $n$-presented ones, the authors show \cite{ep}, using Quillen’s Resolution Theorem, that 
$$\mathsf{K}_i(R) \cong \mathsf{K}_i(\mathsf{FP}_n(R)), \quad i \ge 0,$$
for $R$ a \newterm{left $n$-coherent regular} ring, i.e., a ring in which every finitely $n$-presented module admits a finite resolution by finitely generated projective modules. Here, $\mathsf{FP}_n(R)$ denotes the category of finitely $n$-presented $R$-modules. However, some classical results do not extend directly. For instance, it is well known that $\mathsf{K}_{-1}(R) = 0$ for left coherent regular rings. In contrast, for left $n$-coherent regular rings with $n\ge 2$, this vanishing fails in general: as shown in  Remark \ref{rmk: k-1}, there exists a $2$-coherent regular ring for which $\mathsf{K}_{-1}(R) = \mathbb{Z}.$

These results on $\mathsf{K}_i$ for $n$-coherent regular rings naturally lead to questions about the behavior of $\mathsf{K}_0$ under polynomial extensions. A fundamental problem is to determine when a ring $R$ is $\mathsf{K}_0$-regular, that is, when 
$$\mathsf{K}_0(R[x_1, \dots, x_k]) \cong \mathsf{K}_0(R) \quad \text{for all } k\geq 1.$$
For Noetherian rings, Grothendieck’s Theorem asserts that $R$ is $\mathsf{K}_0$-regular if  every finitely generated $R$-module has finite projective dimension. In the coherent setting, Quillen showed that $\mathsf{K}_0$-regularity holds provided that $R$ is stably coherent (i.e., $R[x_1, \dots, x_k]$ is coherent for all $k>0$) and every finitely presented $R$-module has finite projective dimension. More recently, Swan \cite{Swan} proved in 2019 that for coherent regular rings one has
$$\mathsf{K}_0(R) \cong \mathsf{K}_0(R[x]),$$
but he also observed that, since $R[x]$ need not be coherent even when $R$ is, it remains unclear whether this result extends to polynomial rings in two or more variables $R[x_1, \dots, x_k]$ for $k \ge 2$.  Determining when a ring is coherent is a difficult problem. In this direction, for commutative rings and under the weaker assumption that every finitely $\infty$-presented $R$-module has finite projective dimension— a property we refer to as \newterm{regular} in this work (see Definition \ref{def:regu})—Wang, Qiao, and Kim \cite{Wang} showed that $R$ is $\mathsf{K}_0$-regular. As a consequence, a broad class of rings is $\mathsf{K}_0$-regular, including commutative $n$-coherent regular rings; see Corollary \ref{prop:commutative-K0-regular}.

These results naturally motivate a deeper investigation into the notion of regularity for rings, a concept that has evolved in several forms since the 1960s. The term ``regular'' appears to originate from algebraic geometry: by the classical Zariski--Serre theorem, if $R$ is the local ring of an algebraic variety $V$ at a point $x$, then $R$ is regular if and only if $x$ is a non-singular point of $V$. Furthermore, when $V$ is an affine variety with coordinate ring $B$, $B$ is regular if and only if $V$ itself is non-singular, providing a rich source of commutative regular rings. In a more algebraic framework, a ring $R$ is said to be regular in the sense of Bass if $R$ is left Noetherian and every finitely generated left $R$-module has finite projective dimension. One of the aims of this work is to unify these various notions of regularity and study them from the perspective of finiteness conditions, with the ultimate goal of identifying classes of rings that are $\mathsf{K}_0$-regular.

After setting up the necessary preliminaries in Section \ref{s:preliminaries}, the structure of the paper is as follows. Section \ref{s:Coherence and  von Neumman regularity.} revisits the classical notions of coherence and von Neumann regularity, where we also present new characterizations. In Section~\ref{sec: regularidad}, we investigate the concept of regularity under the assumption of coherence, which allows us to identify a wide family of coherent regular rings. Section \ref{sec: n-coh and regularidad} is concerned with a generalization of this framework: we study and characterize $n$-coherent regular rings, as introduced in \cite{ep}. In Section \ref{sec:Vogel regularidad}, we analyze Vogel’s notion of regularity, which provides an alternative approach to the concept. In Section \ref{sec: uniformly regu} we consider a particularly well-behaved form of regularity, namely uniform $n$-regularity, which exhibits stronger properties than the classical notion. Finally, Section \ref{sec: aplicaciones} contains applications of the preceding results, including examples of rings that are $\mathsf{K}_0$-regular.

\section{Preliminaries}\label{s:preliminaries}
Let \( R \) be an associative ring with identity. We denote by \( \lMod R \) the category of all left \( R \)-modules, and by \( \rMod R \) the category of all right \( R \)-modules. Unless otherwise specified, the terms “module” and “ideal” will always mean left \( R \)-modules and left ideals, respectively.  For an integer \( n \geq 0 \), an \( R \)-module \( M \) is called \newterm{finitely \( n \)-presented} if there exists an exact sequence
$$F_n \longrightarrow F_{n-1} \longrightarrow \cdots \longrightarrow F_1 \longrightarrow F_0 \longrightarrow M \longrightarrow 0,$$
where each \( F_i \) is a finitely generated and projective (or free) \( R \)-module for every \( 0 \leq i \leq n \).  Such a sequence is called a \newterm{finite \( n \)-presentation} of \( M \).  The class of all finitely \( n \)-presented \( R \)-modules is denoted by \( \mathsf{FP}_n(R) \). In particular, \( \mathsf{FP}_0(R) \) consists of all finitely generated \( R \)-modules, and \( \mathsf{FP}_1(R) \) consists of all finitely presented \( R \)-modules. The class \( \mathsf{FP}_\infty(R) \) is defined as the collection of modules that admit a resolution by finitely generated projective (or free) modules. Equivalently, 
$$\mathsf{FP}_\infty(R) = \bigcap_{n \geq 0} \mathsf{FP}_n(R).$$
Clearly, the class of finitely generated projective modules, denoted \( \mathsf{proj}(R) \), is contained in \( \mathsf{FP}_\infty(R) \). We denote by \( \mathsf{Proj}(R) \) the class of all projective \( R \)-modules. For convenience, we also set \( \mathsf{FP}_{-1}(R) := \lMod R \), the entire module category. Moreover, these classes form a chain of inclusions:
$$\mathsf{proj}(R) \subseteq \mathsf{FP}_{\infty}(R) \subseteq \cdots \subseteq \mathsf{FP}_n(R) \subseteq \cdots \subseteq \mathsf{FP}_1(R) \subseteq \mathsf{FP}_0(R)\subseteq \mathsf{FP}_{-1}(R) = \lMod R.$$
The classes \( \mathsf{FP}_n(R) \) have been extensively studied in the literature. In particular, let $0 \rightarrow A \rightarrow B \rightarrow C \rightarrow 0$
be a short exact sequence of \( R \)-modules. Then, by \cite[Theorem 2.1.2]{Glaz1}, the following hold:
\begin{enumerate}
    \item If \( A, B \in \mathsf{FP}_n(R) \), then \( C \in \mathsf{FP}_n(R) \).
    \item If \( A, C \in \mathsf{FP}_n(R) \), then \( B \in \mathsf{FP}_n(R) \).
    \item If \( B \in \mathsf{FP}_n(R) \) and \( C \in \mathsf{FP}_{n+1}(R) \), then \( A \in \mathsf{FP}_n(R) \).
    \item If \( B \cong A \oplus C \), then \( B \in \mathsf{FP}_n(R) \) if and only if \( A, C \in \mathsf{FP}_n(R) \).
\end{enumerate}

As a consequence, the class \( \mathsf{FP}_n(R) \) is closed under extensions, direct summands, and cokernels of monomorphisms between its objects. In contrast, the class \( \mathsf{FP}_{\infty}(R) \) exhibits even stronger closure properties: it is closed under direct summands and forms a \newterm{thick class} of modules; that is, it satisfies the \newterm{2-out-of-3} property with respect to short exact sequences; see \cite[Proposition 1.7 and Theorem 1.8]{BP}. Explicitly, if $0 \rightarrow A \rightarrow B \rightarrow C \rightarrow 0$ is a short exact sequence of \( R \)-modules and two of the modules \( A \), \( B \), or \( C \) lie in \( \mathsf{FP}_{\infty}(R) \), then the third also belongs to this class.

\subsection{Cotorsion pair.}\label{ss:dimensions}
Let \( \mathcal{X} \subseteq \Mod R \). We define
$$\mathcal{X}^{\top_1} \coloneqq \{ M \in \lMod R \mid \Tor^R_1(X, M) = 0 \text{ for all } X \in \mathcal{X} \},$$
$$\mathcal{X}^{\top_{\geq i}} \coloneqq \{ M \in \lMod R \mid \Tor^R_j(X, M) = 0 \text{ for all } X \in \mathcal{X},\, j \geq i \},$$
and for brevity, \( \mathcal{X}^\top \coloneqq \mathcal{X}^{\top_{\geq 1}} \). Dually, for \( \mathcal{Y} \subseteq \lMod R \), we define the classes \( {}^{\top_1}\mathcal{Y} \), \( {}^{\top_{\geq i}}\mathcal{Y} \), and \( {}^{\top}\mathcal{Y} \) in \( \Mod R \) analogously.

Similarly, for \( \mathcal{X} \subseteq \lMod R \), we define
$$\mathcal{X}^{\perp_1} \coloneqq \{ M \in \lMod R \mid \Ext^1_R(X, M) = 0 \text{ for all } X \in \mathcal{X} \},$$
$$\mathcal{X}^{\perp_{\geq i}} \coloneqq \{ M \in \lMod R \mid \Ext^j_R(X, M) = 0 \text{ for all } X \in \mathcal{X},\, j \geq i \},$$
and set \( \mathcal{X}^{\perp} \coloneqq \mathcal{X}^{\perp_{\geq 1}} \). The dual classes \( {}^{\perp_1}\mathcal{Y} \), \( {}^{\perp_{\geq i}}\mathcal{Y} \), and \( {}^{\perp}\mathcal{Y} \) for \( \mathcal{Y} \subseteq \lMod R \) are defined in the same way. If \( \mathcal{X} = \mathsf{FP}_n(R) \), then \( \mathcal{X}^{\perp_1} \) coincides with the class of \( \mathsf{FP}_n \)-injective modules; that is, an \( R \)-module \( M \) is \newterm{\(\mathsf{FP}_n \)-injective} if \( \Ext^1_R(F, M) = 0 \) for all \( F \in \mathsf{FP}_n(R) \), including the case \( n = \infty \). In particular, \( M \) is injective if and only if it is \( \mathsf{FP}_0 \)-injective, and it is \( \mathsf{FP} \)-injective (or absolutely pure) if and only if it is \( \mathsf{FP}_1 \)-injective. The class of \(\mathsf{FP}_n\)-injective modules is denoted by \(\mathsf{FP}_n\text{-}\mathsf{Inj}(R)\). A ring \(R\) is said to be left \newterm{self \(\mathsf{FP}_n\)-injective} if it is \(\mathsf{FP}_n\)-injective as a module over itself. An \(R\)-module \(P\) is \newterm{\(\mathsf{FP}_n\)-projective} if \(\Ext^1_R(P, M) = 0\) for all \(\mathsf{FP}_n\)-injective \(R\)-modules \(M\). The class of \(\mathsf{FP}_n\)-projective modules is denoted by \(\mathsf{FP}_n\text{-}\mathsf{Proj}(R)\). Similarly, if \(\Xcal\) denotes the class of all finitely \(n\)-presented right \(R\)-modules, then \(\Xcal^{\top_{1}}\) coincides with the class of \newterm{\(\mathsf{FP}_n\)-flat} modules. More precisely, an \(R\)-module \(M\) is \newterm{\(\mathsf{FP}_n\)-flat} if \(\Tor_1^R(F, M) = 0\) for every finitely \(n\)-presented right \(R\)-module \(F\) (including the case \(n = \infty\)). Note that $\mathsf{FP}_0$-flat modules coincide with flat modules. Moreover, $\mathsf{FP}_1$-flat modules also coincide with flat modules; see \cite{BP}. Furthermore, \newterm{level modules}, in the sense of \cite{BGH14}, coincide with the $\mathsf{FP}_\infty$-flat modules; that is, those $M \in \lMod R$ such that $\Tor_1^R(F,M)=0$ for all finitely $\mathsf{FP}_\infty$-presented right $R$-modules $F$. The class of \(\mathsf{FP}_n\)-flat modules is denoted by \(\mathsf{FP}_n\text{-}\mathsf{Flat}(R)\).

A \newterm{cotorsion pair} is a pair of classes $(\Acal, \Bcal)$ in $\lMod R$, such that $\Acal = {}^{\perp_1} \Bcal$ and $\Bcal = \Acal^{\perp_1}$. A cotorsion pair is \newterm{hereditary} if moreover $\Acal = {}^{\perp} \Bcal$ and $\Bcal = \Acal^{\perp}$, that is, all the higher $\Ext$-groups vanish as well.  Let $\Xcal$ be a class in $\lMod R$. The cotorsion pair \newterm{generated} by $\Xcal$ is the cotorsion pair $[{}^{\perp_1}(\Xcal^{\perp_1}),\Xcal^{\perp_1}]$. The cotorsion pair \newterm{cogenerated} by $\Xcal$ is the cotorsion pair $[{}^{\perp_1}\Xcal,({}^{\perp_1}\Xcal)^{\perp_1}]$. 

The left-hand class in a cotorsion pair is closed under direct sums and extensions, and contains all the projective modules. The right-hand class in a cotorsion pair is closed under products and extensions, and contains all the injective modules. A class $\Xcal$ in $\lMod R$ is \newterm{resolving} if $\Xcal$ is closed under extensions, contains the projectives, and is closed under kernels of epimorphisms. Dually, a class $\Xcal$ in $\lMod R$ is \newterm{coresolving} if $\Xcal$ is closed under extensions, contains the injectives, and is closed under cokernels of monomorphisms. 
A cotorsion pair $(\Acal, \Bcal)$ is hereditary if and only if $\Acal$ is a resolving class if and only if $\Bcal$ is a coresolving class. A cotorsion pair $(\Acal, \Bcal)$ is \newterm{complete} if $\Acal$ is special precovering, or equivalently, if $\Bcal$ is special preenveloping. In particular, the cotorsion pair \( [\mathsf{FP}_n\text{-}\mathsf{Proj}(R), \mathsf{FP}_n\text{-}\mathsf{Inj}(R)] \) is complete and cogenerated by a set; see \cite[Corollary 4.2]{BP}.

Following \cite{Zhu}, we introduce the notions of dimensions associated with \( \mathsf{FP}_n \)-injective and \( \mathsf{FP}_n \)-flat modules, where \( n \geq 0 \) may also be infinite.

\begin{enumerate}
 \item The \newterm{\( \mathsf{FP}_n \)-flat dimension} of \( N \), denoted by \( \nfd_R(N) \), is the smallest integer \( k \geq 0 \) such that $\Tor^{R}_{k+1}(F, N) = 0$  for every finitely $n$-presented right $R$-module $F$. If no such integer $k$ exists, we define $\nfd_{R}(N) = \infty$.

The \newterm{left \( \mathsf{FP}_n \)-flat global dimension} of \( R \), denoted by \( \gfd(R) \), is defined as:
    $$\gfd(R) = \sup \{ \nfd_R(N) \mid N \text{ is a left \( R \)-module} \}.$$
\item The \newterm{\( \mathsf{FP}_n \)-injective dimension} of an \( R \)-module \( M \), denoted by \( \nid_R(M) \), is the smallest integer \( k \geq 0 \) such that \( \Ext^{k+1}_R(F, M) = 0 \) for all \( F \in \mathsf{FP}_n(R) \). If no such \( k \) exists, \( \nid_R(M) = \infty \).

The \newterm{left \( \mathsf{FP}_n \)-injective global dimension} of \( R \), denoted by \( \gid(R) \), is defined as:
    $$\gid(R) = \sup \{ \nid_R(M) \mid M \text{ is a left \( R \)-module} \}.$$
\end{enumerate}

\section{Pseudo-Coherence, Coherence, and von Neumann Regularity}\label{s:Coherence and  von Neumman regularity.}
In this section, we recall the notions of coherent and pseudo-coherent modules, which generalize the classical Noetherian condition.  An \( R \)-module \( M \) is called \newterm{pseudo-coherent} if every finitely generated submodule of $M$ is finitely presented. Equivalently, $M$ is pseudo-coherent if, for every homomorphism \( R^n \to M \) with \( n < \infty \), the kernel is finitely generated.

\begin{prop}\label{prop:pseudo-coherent-properties}
Let $R$ be a ring. The class of pseudo-coherent $R$-modules is closed under submodules, extensions and direct sums.
\end{prop}
\begin{proof}
    It is straightforward. 
\end{proof}

An \( R \)-module \( M \) is called \newterm{coherent} if it is pseudo-coherent and finitely generated. We denote by \( \mathsf{Coh}(R) \) the class of all coherent \( R \)-modules. For example, over a Noetherian ring, any non-finitely generated projective module is pseudo-coherent but not coherent. A ring \( R \) is called \newterm{left coherent} if \( R \), as an \( R \)-module over itself, is coherent, that is, if every finitely generated ideal of $R$ is finitely presented. The following result has been known since at least the 1960s; for completeness, we include it here and refer to Swan (2019) for a precise statement and proof.

\begin{prop}\cite[Theorem 2.4 and Corollary 2.7]{Swan}\label{prop:coh-subcategory}
Let $R$ be a ring. The class \( \mathsf{Coh}(R) \) of coherent \( R \)-modules is an abelian subcategory of \( \lMod R \). Moreover, \( R \) is left coherent if and only if \(\mathsf{Coh}(R)= \mathsf{FP}_1(R) = \mathsf{FP}_\infty(R)  \).
\end{prop}
\qed

Let $\mathcal{C}$ and $\mathcal{D}$ be classes of $R$-modules. We say that $\mathcal{D}$ has the \newterm{2-out-of-3 property} with respect to $\mathcal{C}$ if, for every short exact sequence $0 \to C_1 \to C_2 \to C_3 \to 0 $, with $C_1, C_2, C_3 \in \mathcal{C}$, if any two of $C_1, C_2, C_3$ belong to $\mathcal{D}$, then the third also belongs to \( \mathcal{D} \).

\begin{prop} \label{prop:2-out-of-3}
Let $R$ be a ring. For $i = 0, 1$, the class $\mathsf{FP}_i(R)$ has the 2-out-of-3 property with respect to the class of pseudo-coherent $R$-modules.
\end{prop}
\begin{proof}
It is enough to observe that in any short exact sequence of $R$-modules
$0 \rightarrow M_0 \rightarrow M_1 \rightarrow M_2 \rightarrow 0,$
if two of $M_0, M_1, M_2$ are coherent, then the third is also coherent by Proposition~\ref{prop:coh-subcategory}.
\end{proof}

The next proposition characterizes coherent rings by replacing the class of finitely presented modules with \( \mathsf{FP}_1 \)-projective, projective, or free modules.

\begin{theorem}\label{prop: coh-proj-fp}
For a ring $R$, the following conditions are equivalent:
\begin{enumerate}
    \item $R$ is left coherent.
    \item Every $\mathsf{FP}_1$-projective $R$-module is pseudo-coherent.
    \item Every projective $R$-module is pseudo-coherent.
    \item Every free $R$-module is pseudo-coherent.
\end{enumerate}
\end{theorem}
\begin{proof}
\((1) \Rightarrow (2)\) follows from \cite[Lemma 1.5]{Pos}. The implications \((2) \Rightarrow (3)\) and \((3) \Rightarrow (4)\) are straightforward. For \((4) \Rightarrow (1)\), note that if every free module is pseudo-coherent, then in particular \( R \) is pseudo-coherent. Hence, every finitely generated ideal of \( R \) is finitely presented, and thus \( R \) is left coherent.
\end{proof}

The following corollary is an immediate consequence of the previous result.

\begin{cor}\label{cor:proj-coh}
For a ring $R$, the following are equivalent:
\begin{enumerate}
    \item \( R \) is left coherent.
    \item All finitely generated projective \( R \)-modules are coherent.
    \item The class of finitely generated projective \( R \)-modules coincides with the class of flat and coherent \( R \)-modules.
\end{enumerate}
\end{cor}
\qed

In an additive category $\mathcal{A}$ with arbitrary direct sums, an object $M$ is said to be \newterm{small} if every morphism  $f \colon M \to \bigoplus_{i \in I} A_i$ factors through a finite direct sum of the $A_i$.  For $R$-modules, $M$ is small if and only if there does not exist an infinite family of submodules  $\{M_j \subsetneq M \mid j \in J\}$ such that every  $m \in M$ lies in almost all  $M_j$; see \cite[Exercise 1B, p. 501]{Lam12}.  It is well known that every finitely generated module is small. Moreover, if $M$ is projective, then $M$ is small if and only if it is finitely generated.

\begin{cor}
    Let $R$ be a ring. Then $R$ is left coherent if and only if every small projective $R$-module is coherent.
\end{cor}
\qed

Recall that an \( R \)-module \( M \) is said to be \newterm{\( \mathsf{f} \)-projective} if, for every finitely generated submodule \( C \subseteq M \), the inclusion map \( C \hookrightarrow M \) factors through a finitely generated free module. Every projective module is \( \mathsf{f} \)-projective, and any finitely generated \( \mathsf{f} \)-projective module is projective. Moreover, since flat modules are direct limits of finitely generated free modules, every \( \mathsf{f} \)-projective module is flat, see \cite[Lemma 3.2]{AT}. The following proposition relates \( \mathsf{f} \)-projectivity to coherence.

\begin{prop}\label{prop: f-projective}
A ring \( R \) is left coherent if and only if every \( \mathsf{f} \)-projective \( R \)-module is pseudo-coherent. In this case, every flat and \( \mathsf{FP}_1 \)-projective \( R \)-module is \( \mathsf{f} \)-projective.
\end{prop}
\begin{proof}
If \( R \) is left coherent, then the result follows directly from \cite[Lemma 3.5]{AT}. Conversely, if every \( \mathsf{f} \)-projective module is pseudo-coherent, then in particular every projective $R$-module is pseudo-coherent. By Theorem \ref{prop: coh-proj-fp}, this implies that \( R \) is left coherent. Finally, by Theorem \ref{prop: coh-proj-fp}, every \( \mathsf{FP}_1 \)-projective \( R \)-module is pseudo-coherent, and by \cite[Lemma 3.5]{AT} again, every flat and \( \mathsf{FP}_1 \)-projective module is \( \mathsf{f} \)-projective.
\end{proof}

\begin{ex}
Over a Prüfer domain, the classes of flat and \( \mathsf{f} \)-projective modules coincide. Moreover, by \cite[Lemma 3.5]{AT}, every flat \( R \)-module is pseudo-coherent. 
\end{ex}

Note that if every flat \( R \)-module is pseudo-coherent, then \( R \) is left coherent by Theorem \ref{prop: coh-proj-fp}.  

\begin{prop}
Let \( R \) be a ring such that every flat \( R \)-module is pseudo-coherent. Then the class of flat \( R \)-modules coincides with the class of \( \mathsf{f} \)-projective modules. In particular, \( R \) contains no infinite set of pairwise orthogonal nonzero idempotents.
\end{prop}
\begin{proof}
Assume that every flat \( R \)-module is pseudo-coherent. Then, by \cite[Lemma 3.5]{AT}, it follows that every flat module is \( \mathsf{f} \)-projective. In particular, every cyclic flat $R$-module is projective. Then, by \cite[Lemma 4.5]{PR04}, \( R \) does not admit an infinite set of pairwise orthogonal nonzero idempotents.
\end{proof}

Recall from \cite{PR04} that a ring \( R \) is called an \newterm{\( \mathsf{S} \)-ring} if every finitely generated flat \( R \)-module is projective.  Since every finitely generated flat module over a local ring is free, it follows that every local ring is an \( \mathsf{S} \)-ring. 
Other examples of left \( \mathsf{S}\)-rings include left Noetherian rings and semiperfect rings; see \cite{PR04}.

\begin{cor} 
Let \( R \) be a ring such that every flat \( R \)-module is pseudo-coherent. Then \( R \) is an \( \mathsf{S} \)-ring.
\end{cor}
\qed

\subsection{Von Neumann regularity}

Recall that a ring \( R \) is called a \newterm{von Neumann regular ring} if, for every element \( r \in R \), there exists \( x \in R \) such that \( r = rxr \). It is well known that \( R \) is von Neumann regular if and only if every \( R \)-module is \( \mathsf{FP}_1 \)-injective. 
Moreover, this condition can be verified by checking it only on ideals of \( R \).

\begin{prop}\label{prop:von Neumann regula-fp-inj}
Let $R$ be a ring. The following conditions are equivalent:
\begin{enumerate}
    \item $R$ is von Neumann regular.
    \item Every nonzero proper ideal of $R$ is $\mathsf{FP}_1$-injective.
    \item Every nonzero principal ideal of $R$ is $\mathsf{FP}_1$-injective.
\end{enumerate}
\end{prop}
\begin{proof}
$(1)\Rightarrow(2)\Rightarrow(3)$ are clear.  

$(3)\Rightarrow(1):$  
It suffices to prove that every principal ideal of $R$ is generated by an idempotent.  
Let $I=(r)$ be a nonzero principal ideal of $R$. Since $I$ is cyclic, the short exact sequence $0 \rightarrow I \rightarrow R \rightarrow R/I \rightarrow 0$
splits, because by assumption $\Ext^1_R(R/I,I)=0$. Hence $I$ is a direct summand of $R$, and therefore $I = Re$ for some idempotent $e \in R$. 
\end{proof}

Recall that a short exact sequence of \( R \)-modules $0 \rightarrow F \rightarrow E \rightarrow G \rightarrow 0$  is called \newterm{pure} if it remains exact under the functor \( \Hom_R(M, -) \) for every finitely presented \( R \)-module \( M \). It is well known that a ring \( R \) is von Neumann regular if and only if every short exact sequence of \( R \)-modules is pure. 

\begin{cor}
Let \( R \) be a ring. Then \( R \) is von Neumann regular if and only if, for every nonzero (principal) ideal \( I \subseteq R \), the short exact sequence $
0 \rightarrow I \rightarrow E(I) \rightarrow E(I)/I \rightarrow 0 $ is pure.
\end{cor}
\qed

For any ring \( R \), the pair \( [\mathsf{FP}_1\text{-}\mathsf{Proj}(R), \mathsf{FP}_1\text{-}\mathsf{Inj}(R)] \) forms a complete cotorsion pair. In particular, every \( R \)-module admits a special \( \mathsf{FP}_1 \)-injective preenvelope and a special \( \mathsf{FP}_1 \)-projective precover; that is, any \( R \)-module can be embedded into an \( \mathsf{FP}_1 \)-injective module with cokernel in \( \mathsf{FP}_1 \)-projectives, and can also be expressed as a quotient of an \( \mathsf{FP}_1 \)-projective module by an \( \mathsf{FP}_1 \)-injective submodule. Moreover, this cotorsion pair is hereditary if and only if \( R \) is a left coherent ring; that is, \( \Ext^i_R(A, B) = 0 \) for all \( i \geq 1 \), whenever \( A \in \mathsf{FP}_1\text{-}\mathsf{Proj}(R) \) and \( B \in \mathsf{FP}_1\text{-}\mathsf{Inj}(R) \).  This follows from \cite[Corollary 4.3]{MD4} together with \cite[Lemma 2.2.10]{Trlifaj}.

We now present a characterization of von Neumann regular rings in terms of the \( \mathsf{FP}_1 \)-injectivity of a specific family of modules constructed from ideals, as shown in the following result.

\begin{theorem}
Let \( R \) be a left coherent ring. For each (principal) ideal \( I \subseteq R \), consider a short exact sequence
$$0 \longrightarrow A_I \longrightarrow B_I \longrightarrow I \longrightarrow 0$$
in \( R\text{-Mod} \), where \( A_I \) is \( \mathsf{FP}_1 \)-injective and \( B_I \) is \( \mathsf{FP}_1 \)-projective. Define the set \( \mathcal{S} := \{ B_I \}_{I \subseteq R} \). Then \( R \) is von Neumann regular if and only if \( \mathcal{S} \subseteq \mathsf{FP}_1\text{-}\mathsf{Inj}(R) \).
\end{theorem}

\begin{proof}
Assume that \( \mathcal{S} \subseteq \mathsf{FP}_1\text{-}\mathsf{Inj}(R) \). Fix a (principal) ideal \( I \subseteq R \) and consider the short exact sequence $0 \rightarrow A_I \rightarrow B_I \rightarrow I \rightarrow 0,$
with \( A_I \in \mathsf{FP}_1\text{-}\mathsf{Inj}(R) \) and \( B_I \in \mathsf{FP}_1\text{-}\mathsf{Proj}(R) \). For any finitely presented \( R \)-module \( K \), the induced long exact sequence in \(\Ext\) gives
$$\cdots \longrightarrow \Ext^1_R(K, A_I) \longrightarrow \Ext^1_R(K, B_I) \longrightarrow \Ext^1_R(K, I) \longrightarrow \Ext^2_R(K, A_I) \longrightarrow \cdots.$$

Since \( A_I \) is \( \mathsf{FP}_1 \)-injective, \(\Ext^1_R(K, A_I) = \Ext^2_R(K, A_I) = 0\), so that
$$\Ext^1_R(K, B_I) \cong \Ext^1_R(K, I).$$

Now, \( B_I \in \mathcal{S} \subseteq \mathsf{FP}_1\text{-}\mathsf{Inj}(R) \) implies \(\Ext^1_R(K, B_I) = 0\), and hence \(\Ext^1_R(K, I) = 0\). Therefore, every (principal) ideal \( I \subseteq R \) is \( \mathsf{FP}_1 \)-injective, and \( R \) is von Neumann regular by Proposition \ref{prop:von Neumann regula-fp-inj}. The converse is immediate.
\end{proof}

\begin{prop} \label{cor:Semisimple}
Let \( R \) be a ring. The following statements are equivalent:
\begin{enumerate}
    \item \( R \) is left Noetherian and von Neumann regular.
    \item Every \( R \)-module is \( \mathsf{FP}_1 \)-projective and flat.
    \item \( R \) is semisimple Artinian.
\end{enumerate}
\end{prop}

\begin{proof}
\((1) \Rightarrow (2):\) 
If \( R \) is left Noetherian, then the classes of injective and \( \mathsf{FP}_1 \)-injective modules coincide, hence every \( R \)-module is \( \mathsf{FP}_1 \)-projective. On the other hand, over a von Neumann regular ring, every module is flat. Therefore, every \( R \)-module is both \( \mathsf{FP}_1 \)-projective and flat.

\((2) \Rightarrow (3):\) Suppose that every \( R \)-module is \( \mathsf{FP}_1 \)-projective and flat. Then \( R \) is left Noetherian, and hence left coherent. Then Proposition \ref{prop: f-projective} shows that every \( R \)-module is \( \mathsf{f} \)-projective. By \cite[Corollary 2.23]{Amzil}, it follows that \( R \) is semisimple Artinian.

\((3) \Rightarrow (1):\) 
It is clear.
\end{proof}

The following result  follows from the previous proposition and \cite[Corollary 4.7]{PR04}.

\begin{cor}\label{cor:S-rings}
    Let \( R \) be a von Neumann regular ring. Then \( R \) is left Noetherian if and only if \( R \) is left \( \mathsf{S} \)-ring.
\end{cor}
\qed

Recall that a ring \( R \) is called left \newterm{semi-Artinian} if every nonzero cyclic \( R \)-module has a nonzero socle.

\begin{cor}
Let \( R \) be a ring. Then the following are equivalent:
\begin{enumerate}
    \item \( R \) is semisimple Artinian.
    \item \( R \) is left semi-Artinian, von Neumann regular, and every simple \( R \)-module is finitely presented.
\end{enumerate}
\end{cor}
\begin{proof}
Follows from Proposition~\ref{cor:Semisimple}  and \cite[Corollary~3.11]{MD6}.
\end{proof}

\section{Regularity and Regular Coherence}\label{sec: regularidad}

\begin{dfn}\label{def:regu}
    Let $R$ be a ring. We say that:
    \begin{enumerate}
     \item \( R \)  is \newterm{left regular} if every finitely \( \infty \)-presented \( R \)-module has finite projective dimension.
     \item  \( R \) is \newterm{left uniformly regular} if there exists an integer \( k \geq 0 \) such that every finitely \( \infty \)-presented \( R \)-module has projective dimension at most \( k \).
    
    \item \( R \) is \newterm{left Bertin-regular} if every finitely generated left ideal of \( R \) has finite projective dimension, see \cite{Bertin71}.
    \end{enumerate}
\end{dfn}

For a ring $R$, we denote by $\mathcal{F}_k(R)$ (resp. $\mathcal{P}_k(R)$) the class of $R$-modules of flat (resp. projective) dimension at most $k$.

\begin{prop}\label{prop: regula flat proj}
    Let \( R \) be a ring. Then:  
    \begin{enumerate}
        \item \( R \) is left regular if and only if every finitely \(\infty\)-presented \( R \)-module has finite flat dimension.  
        \item \( R \) is left uniformly regular if and only if there exists an integer \( k \geq 0 \) such that  $
        \mathsf{FP}_{\infty}(R) \subseteq \mathcal{F}_k(R).$
        In the commutative local case, this is equivalent to  
        $ \Tor_k^R(M, R/\mathsf{m}) = 0$ 
        for every finitely \(\infty\)-presented \( R \)-module \( M \), where \( \mathsf{m} \) is the maximal ideal of \( R \).
  \end{enumerate}
\end{prop}
\begin{proof}
It suffices to show that \( \fd_R(M)=\pd_R(M) \) for every finitely 
\(\infty\)-presented \(R\)-module. Let \(M\) be a finitely \(\infty\)-presented \(R\)-module, and consider a projective resolution 
$$\cdots \longrightarrow P_n \longrightarrow P_{n-1} \longrightarrow \cdots \longrightarrow P_1 \longrightarrow P_0 \longrightarrow M \longrightarrow 0,$$
where each \(P_i\) is finitely generated projective and each syzygy 
\(K_i=\ker(f_{i-1})\) is finitely $\infty$-presented. If \(\fd_R(M)=n<\infty\), then \(K_n\) is finitely presented and flat, hence projective. Thus \(\pd_R(M)\leq n\).
\end{proof}

Classical examples of uniformly regular, regular, or Bertin-regular rings appear among several well-known families. For instance, semihereditary and von Neumann regular rings are uniformly regular, while rings of finite weak dimension are regular.

\begin{prop}\label{prop:regular 0}
Let \( R \) be a ring. Then \( R \) is left regular if and only if every finitely \( \infty \)-presented submodule of a projective \( R \)-module has finite projective dimension.
\end{prop}
\begin{proof}
Since any \( M \in \mathsf{FP}_\infty(R) \) can be expressed as a quotient \( F/N \), where \( F \) is a finitely generated projective \( R \)-module and \( N \in \mathsf{FP}_\infty(R) \), the result follows.
\end{proof}

Inspired by \cite[Theorem 1.7]{Costa}, we can characterize uniform regularity in terms of Ext vanishing between finitely \(\infty\)-presented modules.

\begin{theorem}
    Let \( R \) be a ring. Then \( R \) is left uniformly regular if and only if there exists \( k \geq 0 \) such that \(\Ext_R^{k+1}(M, N) = 0\) for all  finitely $\infty$-presented \( R \)-modules \( M \) and \( N \).
\end{theorem}
\begin{proof}
    We only need to prove that if \( \Ext_R^{k+1}(M, N) = 0 \) for all finitely \(\infty\)-presented modules \( N \), then \( \pd_R(M) \leq k \) for every finitely \(\infty\)-presented module \( M \). For \( k = 0 \), take a short exact sequence \( 0 \to K \to F \to M \to 0 \), where \( F \) is finitely generated and projective, and \( K \) is finitely \(\infty\)-presented. Since \( \Ext_R^1(K, M) = 0 \), the sequence splits and \( M \) is projective.  For general \( k \), apply dimension shifting and argue by induction.
\end{proof}

The following proposition is well known.

\begin{prop}
Let \( R \) be a ring. The following conditions are equivalent:
\begin{enumerate}
    \item Every ideal of \( R \) has finite projective dimension.
    \item Every cyclic \( R \)-module has finite projective dimension.
    \item Every finitely generated \( R \)-module has finite projective dimension.
\end{enumerate}
Moreover, if \( R \) is left Noetherian, then the above are also equivalent to:
\begin{enumerate}
    \setcounter{enumi}{3}
    \item Every finitely generated ideal of \( R \) has finite projective dimension (i.e., \( R \) is left Bertin-regular).
\end{enumerate}
\end{prop}

\begin{proof}
\((1) \Rightarrow (2):\) Is clear.\\
\((2) \Rightarrow (3):\) Let \( M \) be a finitely generated \( R \)-module. By \cite[Lemma 3.5]{Ryo}, there exists a finite sequence of short exact sequences
$$0 \longrightarrow R/I_k \longrightarrow M_{k-1} \longrightarrow M_k \longrightarrow 0 \quad (1 \leq k \leq n),$$
with \( M_0 = M \) and \( M_n = 0 \), for some ideals \( I_k \subseteq R \). Since each \( R/I_k \) has finite projective dimension by hypothesis, it follows that \( \pd_R(M) < \infty \).\\
\((3) \Rightarrow (1):\)  Is clear.\\
\((4) \Leftrightarrow (1):\) If \( R \) is left Noetherian, then every ideal of $R$ is finitely generated.
\end{proof} 

 An example of a regular ring that is not Noetherian is the polynomial ring in countably many variables over a principal ideal domain.

 \begin{rmk} The condition of Bertin-regularity on a ring \( R \), or even the stronger condition that every finitely generated ideal has uniformly bounded projective dimension (specifically bounded by 1), does not guarantee coherence. For example, consider the quiver \( Q \) with three vertices: one arrow \( \alpha \) from vertex 1 to vertex 2, and infinitely many arrows \( \{ \beta_s \mid s \in S \} \) from vertex 2 to vertex 3. Let \( R \) denote the quotient of the path algebra of \( Q \) over a field \( k \) by the ideal generated by \( \{ \alpha \beta_s \mid s \in S \} \). In this case, the supremum of the projective dimensions of finitely generated right ideals of $R$ is 1. Furthermore, \( \gD(R) = 2 \); see \cite[Example 3]{LGO}, yet \( R \) is not right coherent.
 \end{rmk}

We denote the cardinality of a ring \(R\) by \(|R|\).

\begin{prop}
Let \(R\) be a ring with \(|R|\leq \aleph_k\) for some integer \(k \geq 0\). If every finitely generated ideal of \(R\) has finite flat dimension, then \(R\) is left Bertin-regular.
\end{prop}

\begin{proof}
By \cite[Theorem]{S74}, over any ring of cardinality at most \(\aleph_k\), every flat module has projective dimension at most \(k+1\). Let \(I \subseteq R\) be a finitely generated left ideal with \(\fd_R(I) = n < \infty\).  Consider a finite flat resolution
$$0 \longrightarrow F_n \xrightarrow{d_n} F_{n-1} \longrightarrow \cdots 
\longrightarrow F_1 \xrightarrow{d_1} F_0 \xrightarrow{d_0} I \longrightarrow 0,$$ where each \(F_i\) is flat. Since each \(F_i\) admits a finite projective resolution of length at most \(k+1\), the Horseshoe Lemma produces a finite projective resolution of \(I\). Hence \(\pd_R(I)<\infty\), and therefore \(R\) is left Bertin-regular. 
\end{proof}

\begin{cor}
Let \(R\) be a ring with \(|R|\leq \aleph_k\) for some integer \(k\ge 0\), and set \(S=R[x_1,x_2,\dots]\). If every finitely generated ideal of \(S\) has finite flat dimension, then \(S\) is left Bertin-regular. \end{cor}
\begin{proof}
It suffices to observe that $|S| = \max\{|R|,\aleph_0\} = \aleph_k.$
\end{proof}

\begin{rmk}
Let \(R\) be a commutative ring and let \(\mathsf{Q}(R)\) denote its total ring of fractions.  By \cite[Theorem 4.12]{Tar20}, \(|R|=|\mathsf{Q}(R)|\).  Hence, for any \(R\) with \(|R|\leq \aleph_k\), if every finitely generated ideal of \(\mathsf{Q}(R)\) has finite flat dimension, then \(\mathsf{Q}(R)\) is left Bertin-regular.
\end{rmk}

\begin{prop}\label{prop: equi regu}
Let \(R\) be a left coherent ring. The following conditions are equivalent:
\begin{enumerate}
    \item \(R\) is a left regular ring.  
    \item \(R\) is a left Bertin-regular ring.
    \item Every finitely generated ideal of \(R\) has finite flat dimension.
    \item Every cyclic finitely presented  \(R\)-module has finite flat dimension.
\end{enumerate}
\end{prop}
\begin{proof}
The equivalence \((1)\Leftrightarrow(2)\) is given by \cite[Theorem 6.2.1]{Glaz1}. Since \(R\) is left coherent, every finitely generated  ideal is finitely presented, hence finitely \(\infty\)-presented. Thus, by Proposition \ref{prop: regula flat proj}, we obtain \((2)\Leftrightarrow(3)\).   Finally, \((3)\Leftrightarrow(4)\) holds in general, since cyclic finitely presented modules are exactly of the form \(R/I\) with \(I\) finitely generated.
\end{proof}

Recall that an \( R \)-module \( M \) is said to be \newterm{\( \mathcal{C} \)-periodic} if there exists a short exact sequence $0 \rightarrow M \rightarrow C \rightarrow M \rightarrow 0$
with \( C \in \mathcal{C} \). For a ring \( R \), the injectivity of an \( R \)-module \( M \) can be characterized via \cite[Theorem 3.2]{CKWZ}: \( M \) is injective if and only if it is \( \mathsf{Inj}(R) \)-periodic with respect to a pure exact sequence; that is, there exists a pure exact sequence $
0 \rightarrow M \rightarrow C \rightarrow M \rightarrow 0$
with \( C \) an injective module.  However, when  $R$ is left Bertin-regular, the purity assumption can be dropped, as the following result shows:

\begin{prop}
Let \( R \) be a left Bertin-regular ring, and let \( M \) be an \( R \)-module. Then \( M \) is injective if and only if it is \( \mathsf{Inj}(R) \)-periodic. In particular, every \( \mathsf{Inj}(R) \)-periodic module over a left coherent regular ring is injective.
\end{prop}
\begin{proof}
If \(M\) is injective, the canonical short exact sequence  $0 \to M \to M \oplus M \to M \to 0$ shows that \(M\) is \( \mathsf{Inj}(R) \)-periodic. The converse follows from \cite[Proposition 4.8]{BCE}.  
\end{proof}

\begin{rmk}
We denote by \( \mathsf{FP}_1(R)^{<\infty} \) the class of finitely presented \( R \)-modules with finite projective dimension. A module \( M \) belongs to the class \( \varinjlim \mathsf{FP}_1(R)^{<\infty} \) if it can be expressed as a direct limit of modules from \( \mathsf{FP}_1(R)^{<\infty} \). The class \( \varinjlim \mathsf{FP}_1(R)^{<\infty} \) is closed under arbitrary direct sums, pure submodules, pure extensions, and pure epimorphic images; in particular, it is closed under direct limits and forms a covering class; see \cite{PPT}. Moreover, if \( R \) is coherent regular, then \( \varinjlim \mathsf{FP}_1(R)^{<\infty} = \lMod R \), since every module is a filtered colimit of finitely presented modules. 
\end{rmk}

\begin{prop}
    Let \(R\) be a left coherent ring. Then
$$\varinjlim \mathsf{FP_1(R)}^{<\infty} = [^{\top_1}(\mathsf{FP_1(R)}^{<\infty})]^{\top_1}.$$
\end{prop}
\begin{proof}
    Suppose that \( R \) is left coherent. Then \( \mathsf{FP_1(R)}^{<\infty} \subseteq \mathsf{FP_2(R)} \).  
    Since \( \mathsf{FP_1(R)}^{<\infty} \) is closed under extensions and direct summands, it follows from \cite[Lemma 2.1]{PPT} that $\varinjlim \mathsf{FP_1(R)}^{<\infty} =[^{\top_1}(\mathsf{FP_1(R)}^{<\infty})]^{\top_1}.$
\end{proof}

\begin{theorem}
Let \( R \) be a left coherent ring. The following are equivalent:
\begin{enumerate}
    \item Every finitely presented \( R \)-module has finite projective dimension; that is, \( R \) is left regular.
    \item Every finitely generated submodule of a projective \( R \)-module has finite projective dimension.
    \item The kernel of any map between finitely generated projective \( R \)-modules has finite projective dimension.
\end{enumerate}
\end{theorem}
\begin{proof}
(1) $\Leftrightarrow$ (2): Follows from Proposition~\ref{prop:regular 0}.\\
(1) $\Rightarrow$ (3): Assume that $R$ be a left  coherent regular ring. By the corollary \ref{cor:proj-coh}, every finitely generated projective $R$-module is coherent. Thus, for any map \( P_0 \to P_1 \) with \( P_0 \) and \( P_1 \) finitely generated and projective, the kernel \( K \) is coherent, and therefore finitely presented. It follows that \( \pd_R(K) < \infty \).\\
(3) $\Rightarrow$ (1): Let \( I \subseteq R \) be a finitely generated  ideal. Since \( R \) is left coherent, \( I \) is finitely presented, so there exists an exact sequence of $R$-modules $0 \to K \to P_0 \to I \to 0,$ where \( K \) is finitely generated and \( P_0 \) is finitely generated projective. Extending this sequence, we obtain   $0 \to K \to P_0 \to R \to R/I \to 0.$
By assumption, \( \pd_R(K) < \infty \), which implies \( \pd_R(I) < \infty \), and thus \( R \) is left regular by Proposition \ref{prop: equi regu}. 
\end{proof}

It is well known that a ring \(R\) is coherent regular if and only if the matrix ring \(M_n(R)\) is coherent regular for every \(n \geq 1\). This follows from the fact that the categories \(\lMod R \) and $M_n(R)$-$\mathsf{Mod}$ are equivalent, and such equivalences preserve homological properties such as projectivity, exactness, and finite presentation. Following \cite{MD5}, a ring $R$ is  called  left \newterm{$\mathsf{P}$-coherent} if every principal ideal of $R$ is finitely presented. Clearly, every left coherent ring is left $\mathsf{P}$-coherent, but the converse does not hold in general; see \cite[Example 2.3]{MD5}. 

\begin{theorem}
Let \( R \) be a ring. Then \( R \) is left  coherent regular if and only if, for every \( n \geq 1 \), the matrix ring \( M_n(R) \) is left \( \mathsf{P} \)-coherent and every principal  ideal of \( M_n(R) \) has finite projective dimension as an \( M_n(R) \)-module.  
\end{theorem}   
\begin{proof}
The necessity is straightforward, and we thus prove only the sufficiency. Assume that \( M_n(R) \) is left \( \mathsf{P} \)-coherent for every \( n \geq 1 \) and that every principal  ideal of \( M_n(R) \) has finite projective dimension as an \( M_n(R) \)-module. By \cite[Proposition 2.4]{MD5}, \( R \) is left coherent.  

Now, let \( I = Rr_1 + Rr_2 + \cdots + Rr_n \) be a finitely generated  ideal of \( R \). Consider the matrix  
$$A = 
\begin{pmatrix}
    r_1 & 0   & \cdots & 0 \\
    r_2 & 0   & \cdots & 0 \\
    \vdots & \vdots & \ddots & \vdots \\
    r_n & 0   & \cdots & 0
\end{pmatrix}
\in M_n(R),$$
where \( r_1, r_2, \dots, r_n \in R \). Then, the principal left ideal  
$$J = M_n(R) A =
\begin{pmatrix}
    I & 0 & \cdots & 0 \\
    I & 0 & \cdots & 0 \\
    \vdots & \vdots & \ddots & \vdots \\
    I & 0 & \cdots & 0
\end{pmatrix}$$
of \( M_n(R) \) has finite projective dimension by assumption. Therefore, there exists a projective resolution  $0 \to P_k \to P_{k-1} \to \cdots \to P_0 \to J \to 0$
where each \( P_i \) is a projective \( M_n(R) \)-module.  Since \( M_n(R) \) is a free \( R \)-module generated by the \( n^2 \) matrix units, every free \( M_n(R) \)-module is also a free  \( R \)-module. Hence, each \( P_i \) is projective as \( R \)-module, implying that \( J \) has finite projective dimension over \( R \). Since \( J = M_n(R) A \cong I^n \) as \( R \)-modules, it follows that \( I \) also has finite projective dimension.
\end{proof}

Recall that for a ring \( R \), the \newterm{dual} of an \( R \)-module \( M \) is the right \( R \)-module \( M^* = \Hom_R(M, R) \), where the action is given by \( (f \cdot a)(m) = f(m)a \) for all \( f \in M^* \), \( a \in R \), and \( m \in M \). It is well known that an \( R \)-module \( M \) is finitely generated projective of rank \( 1 \) if and only if the canonical map \( M \otimes_R M^* \to R \) is an isomorphism; in this case, \( M^* \) is also finitely generated projective of rank \( 1 \). Moreover, an \( R \)-module \( M \) is said to be \newterm{torsionless} if the natural map \( \phi: M \to \Hom_R(M^*, R) \), given by \( \phi(m)(f) = f(m) \), is injective. Equivalently, \( M \) is torsionless if it embeds into a free \( R \)-module. 

A ring \( R \) is von Neumann regular  if and only if it is left coherent regular and left self \( \mathsf{FP}_1 \)-injective. Although this characterization will be established later in a more general setting; see Proposition \ref{prop: coherencia self fp}, we can already deduce the following consequence from \cite[Proposition 2.8 and Example 2.9(2)]{MD1} and the previous result.

\begin{cor}  
Let \( R \) be a ring. Then \( R \) is von Neumann regular if and only if, for every \( n \geq 1 \), the following conditions hold:  
\begin{enumerate}  
    \item Every finitely presented cyclic right \( M_n(R) \)-module is torsionless.  
    \item The matrix ring \( M_n(R) \) is left \( \mathsf{P} \)-coherent, and every principal ideal of \( M_n(R) \) has finite projective dimension as an \( M_n(R) \)-module.  
\end{enumerate}  
\end{cor}
\qed

Recall that an ideal \( I \) of \( R \) is called a \newterm{\( \mathsf{W} \)-ideal} if it is isomorphic to a quotient of the dual module of a finitely generated right \( R \)-module. By \cite[Theorem 1]{Wang93}, a ring \( R \) is left coherent if and only if every \( \mathsf{W} \)-ideal and finitely generated ideal of \( R \) is finitely presented. This allows us to characterize coherent regular rings in terms of \( \mathsf{W} \)-ideals.

\begin{prop}
Let \( R \) be a ring. Then \( R \) is left coherent regular if and only if every \( \mathsf{W} \)-ideal and finitely generated ideal of \( R \) is finitely presented and has finite projective dimension.
\end{prop}
\begin{proof} We only need to prove the converse. Assume that every \( \mathsf{W} \)-ideal and finitely generated  ideal of \( R \) is finitely presented and has finite projective dimension. Then \( R \) is left coherent by \cite[Theorem 1]{Wang93}. Now, let \( I \subseteq R \) be a finitely generated left ideal. Since \( I \) is finitely presented, there exists a short exact sequence  of $R$-modules
$ 0 \to K \to R^n \to I \to 0,$
where \( K \) is finitely generated. Thus, \( I \cong R^n/K \). Note that \( R^n \cong \Hom_R(R^n_R, R) \), where \( R^n_R \) denotes \( R^n \) considered as a right \( R \)-module. Hence, \( I \) is a \( \mathsf{W} \)-ideal. By assumption, \( I \) has finite projective dimension. Therefore, \( R \) is left regular.
\end{proof}

Recall that the \newterm{small finitistic projective dimension} of a ring $R$, denoted by $\sfindim (R)$, is equal to the supremum of the projective dimensions of $R$-modules $M$, which satisfy $\pd_R M < \infty$, and $M$ is of type $\mathsf{FP}_{\infty}(R)$. 

\begin{prop}
Let \( R \) be a ring. Then the following hold:
\begin{enumerate}
    \item If \( R \) is a left regular ring, then the following are identical.
    \begin{enumerate}[label=(\roman*)]
        \item \( \sfindim(R) \).
        \item \( \gidi(R) \).
        \item \( \gfdi(R^{\op}) \).
    \end{enumerate}
     \item If \( R \) is a left coherent regular ring, then the following are identical.
    \begin{enumerate}[label=(\roman*)]
        \item \( \gidone(R) \).
        \item  \( \sfindim(R) \).
        \item  \( \displaystyle\sup\{\pd_R(M) : M \in \mathsf{FP}_1\text{-}\mathsf{Proj}(R),\ \pd_R(M) < \infty\} \).
       \item   \( \displaystyle\sup\{\pd_R(M) : M \in \mathsf{FP}_1\text{-}\mathsf{Proj}(R)\} \).
        \item \( \displaystyle\sup\{\oneid_R(M) : M \in \mathsf{FP}_1\text{-}\mathsf{Proj}(R)\} \).
        \item \( \wD(R) \).
    \end{enumerate}
\end{enumerate}
\end{prop}
\begin{proof}
Part (1) follows from \cite[Theorem 3.8]{Gao}.  

For part (2), assume that \( R \) is a left coherent regular ring.  
\textbf{(i) $=$ (ii):} by \cite[Lemma 1.1]{AM}.  
\textbf{(i) $=$ (iv) $=$ (v):} by \cite[Theorem 2.5]{AM}.  
\textbf{(i) $=$ (vi):} well known.  
\textbf{(iii) $=$ (vi):} observe that
$$\begin{aligned}
    \wD(R) 
    &= \sup\{\pd_R(M) : M \in \mathsf{FP}_1(R)\} \\
    &= \sup\{\pd_R(M) : M \in \mathsf{FP}_1(R), \, \pd_R(M) < \infty\} \\
    &\le \sup\{\pd_R(M) : M \in \mathsf{FP}_1\text{-}\mathsf{Proj}(R), \, \pd_R(M) < \infty\} \\
    &\le \sup\{\pd_R(M) : M \in \mathsf{FP}_1\text{-}\mathsf{Proj}(R)\} \\
    &= \wD(R).
\end{aligned}$$
 \end{proof}

\begin{cor}
Let $R$ be a ring. The following conditions are equivalent:
\begin{enumerate}
    \item $R$ is von Neumann regular.
    \item $R$ is left coherent, $\wD(R) < \infty$, and every $\mathsf{FP}_1$-projective $R$-module with finite projective dimension is projective.
\end{enumerate}
\end{cor}
\qed

\begin{cor}
Let $R$ be a ring. The following conditions are equivalent:
\begin{enumerate}
    \item $R$ is left semihereditary.
    \item $R$ is left coherent, $\wD(R) < \infty$,  and \( \sfindim(R)\le 1 \).
    \item $R$ is left coherent regular with \( \sfindim(R)\le 1 \).
\end{enumerate}
\end{cor}
\qed

\subsection{Commuatative case}
Recall that an ideal of \(R\) is \newterm{regular} if it contains a regular element, i.e., a non-zerodivisor element. In the context of rings containing regular elements, a ring \(R\) is coherent if and only if every finitely generated ideal of \(R\) that is contained in a finitely presented proper ideal of \(R\) is itself finitely presented; see \cite[Proposition 2.2]{BM}. The same argument as in the proof of \cite[Proposition 2.2]{BM} shows that a ring \(R\) is Bertin-regular if and only if any finitely generated ideal of \(R\) contained in a finitely presented proper ideal of \(R\) has finite projective dimension. Thus, we have the following proposition:

\begin{prop}
    Let \(R\) be a commutative ring containing a regular element. Then \(R\) is  coherent regular  if and only if any finitely generated ideal of \(R\) contained in a finitely presented proper ideal of \(R\) is itself finitely presented and has finite projective dimension.
\end{prop}
\qed

Recall that a domain \( R \) is a \newterm{\(\mathsf{G\text{-}GCD}\) domain} (Generalized Greatest Common Divisor) if the intersection of any two invertible ideals of \( R \) is itself invertible. A fractional ideal \( I \) of \( R \) is \newterm{invertible} if \( I I^{-1} = R \). It is well known that \( I \) is invertible if and only if it is a projective \( R \)-module containing a non-zerodivisor. Following Glaz \cite{Glaz6}, a ring \(R\) is a $\mathsf{G\text{-}GCD}$ ring if every principal ideal is projective and the class of finitely generated flat (equivalently, projective or invertible) ideals is closed under finite intersections. Note that the classes of coherent rings and \(\mathsf{G\text{-}GCD}\) rings are not comparable.  However, every  coherent regular ring is a \(\mathsf{G\text{-}GCD}\) ring; see \cite[Proposition 3.4]{Glaz6}. The following implications always hold: 
$$\text{Coherent regular rings}\Rightarrow \mathsf{G\text{-}GCD}\text{ rings} \Rightarrow \mathsf{p.p.}\text{-rings} \Rightarrow \mathsf{p.f.}\text{-rings} $$
where \(\mathsf{p.p.}\)-rings and \(\mathsf{p.f.}\)-rings are those whose principal ideals are projective and flat, respectively.

\begin{prop}\label{prop: red}
    Let \( R \) be a commutative coherent regular ring. Then:
    \begin{enumerate}
        \item \( R \) is a reduced ring, i.e., it has no non-zero nilpotent elements.
        \item \( R \) is a nonsingular ring, i.e., every essential ideal of \( R \) has zero annihilator.
        \item For each \( r \in R \), \( \mathsf{Ann}_R(r) = \mathsf{Ann}_R(r^k) \) for all \( k \geq 2 \), and \( \mathsf{Ann}_R(r) \) is a pure ideal of $R$ generated by an idempotent element of \( R \).
        \item For each ideal \( I \) of \( R \) and each \( r \in R \), \( \mathsf{Ann}_R(r) \cap I = \mathsf{Ann}_R(r) I \).  
    \end{enumerate}
\end{prop}
\begin{proof}  
     Since \( R \) is a \(\mathsf{G\text{-}GCD}\) ring, its principal ideals are projective, which implies that \( R \) is reduced. In commutative rings, reduced and nonsingular properties are equivalent, so (2) follows from (1).  
     For (3), we have \( \mathsf{Ann}(r) = \mathsf{Ann}(r^k) \) by (1). Additionally, by \cite[Corollary 2.3]{Tar}, \( \mathsf{Ann}(r) \) is a pure ideal of $R$, and by \cite[Corollary 3.3]{Tar}, it is generated by an idempotent element.  
     Finally, (4) follows from (3).  
\end{proof}

\begin{rmk}
For any element \( r \in R \), define \( \mathsf{Supp}(r) = \{ P \in \mathsf{Spec}(R) : r \notin P \} \). Here, \( \mathsf{Spec}(R) \) denotes the set of all prime ideals of \( R \). For any (not necessarily commutative) ring \( R \), the collection \( \{ \mathsf{Supp}(r) : r \in R \} \) forms a basis for the open sets of the topological space \( \mathsf{Spec}(R) \).  We view \( \mathsf{Min}(R) \), the set of minimal prime ideals of \( R \), as a subspace of \( \mathsf{Spec}(R) \) with the induced topology. This topology is known as the \newterm{hull-kernel} topology, and in the commutative setting, it coincides with the Zariski topology. In the commutative case, coherent regular rings are reduced and have compact \( \mathsf{Min}(R) \); see \cite[Proposition 3.1]{Glaz6}. Moreover, \( \mathsf{Min}(R) \) is compact for every reduced left coherent ring, see \cite[Corollary 2.13]{DT21}. However, we do not know whether every left coherent regular ring is reduced, or whether \( \mathsf{Min}(R) \) is compact in general.
\end{rmk}

By Goldie's Theorem, a ring has a semisimple total ring of quotients if it is semiprime, satisfies the ascending chain condition on annihilators, and has no infinite direct sums of nonzero ideals. For the case of commutative coherent regular rings, we now have:

\begin{prop}
Let \( R \) be a commutative coherent regular ring. If every flat \( R \)-module is \( \mathsf{FP} \)-projective, then the total ring of quotients \( \mathsf{Q}(R) \) is semisimple Artinian. In particular, the total ring of quotients of every commutative Noetherian regular ring is semisimple Artinian.
\end{prop}
\begin{proof}
Follows from \cite[Corollary 3.2]{BCS} and Proposition \ref{prop: f-projective}. In the Noetherian case, it suffices to observe that every module is $\mathsf{FP}$-projective.
\end{proof}

\begin{rmk}
    Coherent rings with finite weak global dimension are regular. Von Neumann regular and semihereditary rings are examples. However, there are coherent rings  with infinite weak global dimension which are regular. For example, let $\mathsf{k}$ be a field and $\mathsf{k}[x_1,\cdots,x_n]$ be the polynomial ring in $n$ variables over $\mathsf{k}$. Then $R=\varinjlim_{n\leq \omega}\mathsf{k}[x_1,\cdots,x_n]$ is a coherent regular domain with $\wD(R)=\infty$, see \cite{Glaz1}. However, if a local coherent ring \( (R, \mathsf{m}) \) has quasi-compact punctured spectrum \( \mathsf{Spec}(R) \setminus \{\mathsf{m}\} \), then the regularity condition implies that \( R \) has finite weak dimension; see \cite[Theorem 2.17]{GL}. In particular, any local coherent ring with quasi-compact punctured spectrum and nonzero nilpotent maximal ideal cannot be regular \cite[Proposition 6.3]{BG}.
\end{rmk}

\begin{prop}\label{prop:local regula}
    Let $(R,\mathsf{m})$ be a local coherent regular ring. Then $\mathsf{m}^2\neq 0$. 
\end{prop}
\begin{proof}
If $\mathsf{m}^2 = 0$, then for any nonzero $x \in \mathsf{m}$ we have $\mathsf{Ann}_R(x) = \mathsf{m}$, so by Proposition \ref{prop: red}, $\mathsf{m}$ is generated by an idempotent, a contradiction.
\end{proof}

\begin{prop}
    Let \( R \) be a commutative coherent regular ring. If \( R \) is indecomposable, then \( R \) is a domain.
\end{prop}
\begin{proof}
    Let \( a \in R \) be nonzero. Since \(R\) is coherent regular the ideal \(aR\) is projective, so the sequence $0 \rightarrow \mathsf{Ann}_R(a) \rightarrow R \rightarrow aR \rightarrow 0$
    splits. Hence \( \mathsf{Ann}_R(a) \) is a direct summand of \( R \). If \( R \) is indecomposable, it follows that \( \mathsf{Ann}_R(a) = 0 \), and thus \( a \) is a non-zerodivisor. Therefore, \( R \) is a domain.
\end{proof}

It is well known that in a coherent integral domain, every prime ideal whose power is finitely generated is itself finitely generated. This property, however, does not extend to arbitrary coherent rings. In \cite{Cou5}, it is shown that in a coherent ring, any maximal or non-minimal prime ideal with a finitely generated power is also finitely generated; see \cite[Theorem 3.3]{Cou5}. In the case of regular coherent rings, we have the following result:

\begin{prop}[\cite{Cou5}, Corollary 4.5]
    Let \( R \) be a commutative coherent regular ring. Then every prime ideal with a finitely generated power is also finitely generated.
\end{prop}
\qed

Let \( R \) be a commutative ring, and let \( x \) be an indeterminate over \( R \). For a polynomial \( f \in R[x] \), the \newterm{content} of \( f \), denoted by \( c(f) \), is the ideal of \( R \) generated by the coefficients of \( f \). For any two polynomials \( f, g \in R[x] \), it holds that \( c(fg) \subseteq c(f)c(g) \). A polynomial \( f \) is called a \newterm{Gaussian polynomial} if this inclusion is an equality for every polynomial \( g \in R[x] \), i.e., \( c(fg) = c(f)c(g) \). It is known that \( R \) is a Gaussian reduced ring if and only if \( \wD(R) \leq 1 \). Moreover, by \cite[Theorem 3.3]{Glaz4}, every coherent Gaussian ring satisfies \( \wD(R) \leq 1 \) or \( \wD(R) = \infty \). As a consequence, we obtain:

\begin{prop}  
Let \( R \) be a  Gaussian ring. Then \( R \) is  coherent regular if and only if \( R \) is either von Neumann regular or semihereditary.  
\end{prop}
\qed

\begin{rmk}
Let \( R \) be a Gaussian ring (not necessarily coherent). By \cite[Proposition 5.3]{BG}, if \( R \) is regular, then it is uniformly regular. 
\end{rmk}

 Given an \( R \)-module \( M \), let \( \mathsf{S}^0 M \) denote the complex concentrated in degree zero with \( M \) in degree zero and zero elsewhere. The \newterm{derived category} \( \mathsf{D}(R) \) is obtained from the category of chain complexes of \( R \)-modules by formally inverting quasi-isomorphisms. An object \( X \in \mathsf{D}(R) \) is called \newterm{small or compact} if the canonical map
\[
\bigoplus_{i \in I} \Hom_{\mathsf{D}(R)}(X, Y_i) \longrightarrow \Hom_{\mathsf{D}(R)}\left(X, \bigoplus_{i \in I} Y_i\right)
\]
is an isomorphism for every family \( \{Y_i\}_{i \in I} \) in \( \mathsf{D}(R) \); that is, \( \Hom_{\mathsf{D}(R)}(X, -) \) commutes with arbitrary coproducts. It is well know that a complex of $R$-modules is quasi-isomorphic to a perfect complex if and only if it is a small object of $\mathsf{D}(R)$.  According to \cite[Lemma 3.3]{Hov}, if the complex \( \mathsf{S}^0M \) is small in \( \mathsf{D}(R) \), then \( M \) has finite projective dimension. Moreover, if \( R \) is coherent and \( M \) is a finitely presented $R$-module with finite projective dimension, then \( \mathsf{S}^0M \) is a small object in \( \mathsf{D}(R) \) \cite[Proposition 3.4]{Hov}. 

\begin{prop}\cite[Lemma 3.3 and Proposition 3.4]{Hov}
    Let \( R \) be a commutative coherent ring. Then \( R \) is regular if and only if \( \mathsf{S}^0M \) is a small object in the derived category \( \mathsf{D}(R) \) for every finitely presented \( R \)-module \( M \).
\end{prop}
\qed

\subsubsection{Semi-Regularity}
Matlis introduced an alternative notion of regularity, defining a ring to be \newterm{semi-regular} if every module embeds into a flat module, or equivalently, if every injective module is flat. It is well known that von Neumann regular rings are semi-regular, and that every semi-regular ring is coherent and self \(\mathsf{FP}_1\)-injective.

\begin{prop}\label{prop: semi-regular}
    Let \( R \) be a commutative regular ring. Then \( R \) is semi-regular if and only if it is von Neumann regular.
\end{prop}
\begin{proof}
    Assume that \( R \) is semi-regular, and let \( I \) be a finitely generated ideal of \( R \). By \cite[Proposition 1.3]{BC}, the module \( R/I \) embeds into a free module. Since \( R \) is a $\mathsf{p.p.}$-ring, it follows from \cite[Lemma 15]{Cou3} that \( R/I \) is projective. Hence \( I \) is a direct summand of \( R \). The converse is clear.
\end{proof}

Recall that if \( A \) is a ring and \( E \) is an \( A \)-module, the \newterm{trivial ring extension} of \( A \) by \( E \), denoted \( A \ltimes E \), is defined as the additive group \( A \times E \) with multiplication given by $(a, e)(a', e') = (aa', ae' + a'e).$  The following example shows that the regularity assumption in Proposition \ref{prop: semi-regular} is essential.

\begin{ex}
Let \( A \) be an integrally closed Prüfer domain that is not a field (e.g., \( A := \mathbb{Z} + X \mathbb{Q}[X] \)), and consider the trivial ring extension \( R := A \ltimes \frac{Q(A)}{A} \), where \( Q(A) \) denotes the field of fractions of \( A \). Then, by \cite[Example 2.12]{AK}, the ring \( R \) is semi-regular, but not von Neumann regular.
\end{ex}

\begin{cor}
    Let \( R \) be a commutative coherent regular ring. The following conditions are equivalent:
    \begin{enumerate}
         \item \( R \) is a von Neumann regular ring.
        \item \( R \) is a semi-regular ring.
        \item \( R \) is self \(\mathsf{FP_1}\)-injective ring.
    \end{enumerate}
\end{cor}
\begin{proof}
    Follows from Proposition \ref{prop: semi-regular} and \cite[Theorem II.5]{Cou2}.
\end{proof}

\begin{cor}  
    Let \( R \) be a commutative coherent regular ring. Then \( R \) is a von Neumann regular ring if and only if \( \mathsf{Ann}_R(\mathsf{Ann}_R(I)) = I \) for every finitely generated ideal \( I \) of \( R \).  
\end{cor}  
\begin{proof}
    Follows from  Proposition \ref{prop: semi-regular}  and \cite[Proposition 1.3]{BC}.
\end{proof}  

\subsubsection{Stably Coherent Rings}
In several ways, coherent rings differ significantly from Noetherian rings, most notably in the failure of the "basis theorem". Specifically, if a ring \( R \) is coherent, the polynomial ring \( R[x] \) in one indeterminate over \( R \) is not necessarily coherent. For instance, Soublin showed that \( R = \prod_{i \in \mathbb{N}} \mathbb{Q}[[x, y]] \) is coherent, but \( R[x] \) is not; see \cite[Example 7.3.13]{Glaz1}. A ring \( R \) is called \newterm{\( n \)-stably coherent} if \( R[x_1, \dots, x_n] \) is coherent, and \newterm{stably coherent} if this holds for every \( n \geq 1 \). In fact, if \( R \) is \( n \)-stably coherent for all \( n \), then the polynomial ring in countably many indeterminates \( R[x_1, x_2, \dots] \) is also coherent \cite[Example 2.3.3]{Glaz1}. The class of stably coherent rings is stable under localization, quotients by finitely generated ideals, and finitely presented extensions \cite[Theorems 2.4.1–2.4.2]{Glaz1}. Moreover, any coherent ring of global dimension at most \( 2 \), or of weak global dimension at most \( 1 \) (i.e., semihereditary), is stably coherent. On the other hand, a commutative ring of global dimension \( 2 \) need not be coherent.

\begin{prop}\label{prop:domain poli}
Let \( R \) be a domain with global dimension \( 2 \). Then \( R \) is stably coherent.
\end{prop}
\begin{proof}
By \cite[Corollary 2 and Remark]{Yicai}, \( R \) is either a Prüfer domain or a Noetherian domain. 
\end{proof}

Coherent rings with weak global dimension \( 2 \), including certain domains, need not be stably coherent; see \cite[Theorem 7.3.13]{Glaz1}.

\begin{ex}
\begin{enumerate}
\item Let $R$ be a commutative ring  and $G$ an abelian multiplicative group. The \newterm{group ring} $RG$ is the free $R$-module with basis $G$ and multiplication induced by $G$, where each $x \in RG$ has a unique expression $x = \sum_{g \in G} x_g g$ with finitely many $x_g \neq 0$. The \newterm{rank} of $G$, denoted $\mathsf{rank} G$, is defined as follows:   A set of non-identity elements \( \{ g_1, \dots, g_k \} \) in \( G \) is calle \newterm{independent} if the equation $g_1^{n_1} g_2^{n_2} \cdots g_k^{n_k} = 1,$ where \( 0 < n_i \in \mathbb{Z} \), implies that  $g_1^{n_1} = g_2^{n_2} = \cdots = g_k^{n_k} = 1.$  An infinite set of elements of \( G \) is called \newterm{independent} if every finite subset of it is independent.  For every group \( G \), there exists a maximal independent set consisting only of elements of infinite order. The cardinality of this set is $\mathsf{rank} G$. Following \cite[Theorem 8.2.4]{Glaz1}, $RG$ is coherent if and only if $R$ is $n$-stably coherent when $0 < \mathsf{rank} G = n < \infty$, and if $\mathsf{rank} G = \infty$, then $RG$ is coherent if and only if $R$ is stably coherent.

\item A ring \( R \) is called a \newterm{valuation ring} if the \( R \)-module \( R \) is uniserial; that is, for all \( x, y \in R \), either \( x \in Ry \) or \( y \in Rx \). Equivalently, \( R \) is local and every finitely generated ideal is principal. If the set \( \mathsf{Z} \) of zero divisors of \( R \) is trivial (i.e., \( \mathsf{Z} = 0 \)), then \( R \) is a domain. In this case, \( R \) is stably coherent; see \cite[Theorem 7.3.3]{Glaz1}.  Moreover, any smooth algebra over a valuation ring is also stably coherent; see \cite[Corollary 2.3]{ant2}. Although frequently non-Noetherian, valuation rings exhibit many characteristics similar to Noetherian regular rings.
\end{enumerate}
\end{ex}

\begin{prop}
    Let \( (R,\mathsf{m}) \) be a local commutative ring with \( \gD(R) \leq 2 \). Then \( R \) is stably coherent.
\end{prop}
\begin{proof}
    By \cite{Vas}, local rings of global dimension at most two are precisely the following: two-dimensional regular Noetherian local rings, valuation rings, and umbrella rings. Each of these is a coherent domain, and therefore, by Proposition \ref{prop:domain poli}, \( R \) is stably coherent.
\end{proof}

There exist commutative local rings \( R \) with \( \glD(R) = 3 \) that are not coherent, see \cite[Remark 2.15(4)]{CKWZ}.

\begin{prop}\label{prop:reg poli}
    Let \( R \) be a commutative ring. If \( R \) is regular, then \( R[x_1, \dots, x_n] \) is regular for every \( n \geq 1 \).
\end{prop}
\begin{proof}
    It suffices to prove the case \( n = 1 \), as the general case follows by induction. Let \( N \) be a finitely $\infty$-presented \( R[x] \)-module. Then there is a short exact sequence \( 0 \to M \to F[x] \to N \to 0 \), where \( F \) is a finitely generated free \( R \)-module and \( M \) is finitely $\infty$-presented. Since \( \pd_{R[x]}(N) < \infty \) if and only if \( \pd_{R[x]}(M) < \infty \), so it suffices to prove the claim for \( M \). By \cite[Lemma 3.3]{Wang}, there is an exact sequence \( 0 \to A[x] \to B[x] \to M \to 0 \) with \( A \), \( B \)  finitely $\infty$-presented \( R \)-modules. Since \( R \) is regular, we have \( \pd_R(A) < \infty \) and \( \pd_R(B) < \infty \), and hence \( \pd_{R[x]}(A[x]) < \infty \) and \( \pd_{R[x]}(B[x]) < \infty \). Therefore, \( \pd_{R[x]}(M) < \infty \), and \( R[x] \) is regular.
\end{proof}

If \( R \) is a coherent regular ring, then \( R[x] \) is a \( \mathsf{G\text{-}GCD} \) ring by \cite[Theorem 4.3]{Glaz6}. Moreover, \( R \) is quasi-regular, i.e., its total ring of quotients \( \mathsf{Q}(R) \) is von Neumann regular \cite[Proposition 3.1]{Glaz6}, and thus \( \mathsf{Q}(R)[x] \) is semihereditary. Consequently, \( \mathsf{Q}(R)[x] \) is stably coherent.

\begin{cor}
    Let \( R \) be a commutative coherent regular ring with nonzero nilradical, and set \( R_{\mathrm{red}} := R / \mathrm{Nil}(R) \). 
    If \( \mathrm{Nil}(R) \) is finitely presented and \( \wD(R_{\mathrm{red}}) \leq 1 \), then \( R[x_1,\dots,x_n] \) is coherent regular for all \( n \geq 1 \).
\end{cor}
\begin{proof}
    By Proposition \ref{prop:reg poli}, \( R[x_1,\dots,x_n] \) is regular for all \( n \geq 1 \). Since \( \mathrm{Nil}(R) \) is finitely generated, it follows that \( R_{\mathrm{red}} \) is coherent. Moreover, the assumption \( \wD(R_{\mathrm{red}}) \leq 1 \) implies that \( R_{\mathrm{red}} \) is semihereditary, and hence    \( R[x_1,\dots,x_n] \) is coherent for all \( n \geq 1 \) by \cite[Corollary 4.8]{Asc}. 
\end{proof}

\begin{rmk}
Let \( \alpha : \mathbb{N} \to \mathbb{N} \) be a function. A finitely generated \( R \)-module \( M \) is said to be \newterm{\( \alpha \)-uniformly coherent} if, for every \( n \in \mathbb{N} \), the kernel of every \( R \)-module homomorphism \( R^n \to M \) is generated by at most \( \alpha(n) \) elements. A ring \( R \) is called \newterm{uniformly coherent} if it is \( \alpha \)-uniformly coherent for some function \( \alpha : \mathbb{N} \to \mathbb{N} \). Clearly, every uniformly coherent ring is coherent. Moreover, by \cite[Lemma 3.8]{Asc}, if \( R \) is uniformly coherent and \( R[x_1,\dots,x_n]^{\mathbb{N}} \) is flat over \( R^{\mathbb{N}}[x_1,\dots,x_n] \) for every \( n \geq 1 \), then \( R \) is \( n \)-stably coherent for all \( n \geq 1 \).
\end{rmk}

\section{$n$-Coherence and Regularity}\label{sec: n-coh and regularidad}

Let \( n \geq 0 \) be a non-negative integer, and let \( R \) be a ring. A finitely \( n \)-presented \( R \)-module \( M \) is said to be \newterm{\( n \)-coherent} if every finitely \( (n{-}1) \)-presented submodule of \( M \) is finitely \( n \)-presented. We denote by \( \mathsf{Coh}_n(R) \) the class of all \( n \)-coherent \( R \)-modules. Note that \( \mathsf{Coh}_0(R) \) coincides with the class of Noetherian $R$-modules, and \( \mathsf{Coh}_1(R) \) with the class of coherent $R$-modules. The ring \( R \) is called \newterm{left weakly \( n \)-coherent} if every finitely \( (n{-}1) \)-presented left ideal is finitely \( n \)-presented, and \newterm{left \( n \)-coherent} if every finitely \( n \)-presented left \( R \)-module is finitely \( (n{+}1) \)-presented. For \( n = 0 \) and \( n = 1 \), both notions coincide and recover the classical definitions of Noetherian and coherent rings, respectively. For \( n \geq 2 \), it remains an open problem whether weak \( n \)-coherence implies \( n \)-coherence.

It is well known that a ring \( R \) is left \( n \)-coherent if and only if \( \mathsf{FP}_n(R) = \mathsf{FP}_\infty(R) \), in which case we also have \( \mathsf{FP}_n(R) = \mathsf{Coh}_n(R) \); see \cite[Proposition 2.5]{ep}. In particular, when \( n = 1 \), the class \( \mathsf{Coh}_1(R) \) is abelian, so \( \mathsf{FP}_1(R) \) is an abelian category if and only if \( R \) is left coherent. More generally, \( \mathsf{FP}_n(R) \) is abelian for some \( n \geq 1 \) if and only if \( R \) is left coherent; see \cite[Proposition 5.1]{ep}.

\begin{theorem}
Let \( R \) be a left \( (n+1) \)-coherent ring. Then \( R \) is left \( n \)-coherent if and only if every finitely \( n \)-presented \( R \)-module embeds into a finitely \( (n+1) \)-presented \( R \)-module.
\end{theorem}
\begin{proof}
Suppose \( R \) is left \( (n+1) \)-coherent, and every module in \( \mathsf{FP}_n(R) \) embeds into one in \( \mathsf{FP}_{n+1}(R) \). Let \( M \in \mathsf{FP}_n(R) \), there exists a monomorphism \( M \hookrightarrow P \) with \( P \in \mathsf{FP}_{n+1}(R) \). Since \( R \) is \( (n+1) \)-coherent, it follows from \cite[Proposition~2.5]{ep} that \( \mathsf{FP}_{n+1}(R) = \mathsf{Coh}_{n+1}(R) \), so \( P \) is \( (n+1) \)-coherent. In particular, every finitely \( n \)-presented submodule of \( P \), and thus \( M \), belongs to \( \mathsf{FP}_{n+1}(R) \). Conversely, if \( R \) is left \( n \)-coherent, then \( \mathsf{FP}_n(R) = \mathsf{FP}_{n+1}(R) \), and each module in \( \mathsf{FP}_n(R) \) trivially embeds into itself.
\end{proof}

\begin{prop}
    Let \( R \) be a left \( n \)-coherent ring. Then:
    \begin{enumerate}
        \item $\mathsf{FP}_n\text{-}\mathsf{Inj}(R)\cap [\mathsf{FP}_n\text{-}\mathsf{Inj}(R)]^{\perp_{1}}=\mathsf{Inj}(R)$.
        \item $\mathsf{FP}_n\text{-}\mathsf{Proj}(R)\cap {}^{\perp_1}[\mathsf{FP}_n\text{-}\mathsf{Proj}(R)]=\mathsf{Proj}(R)$
    \end{enumerate}    
\end{prop}
\begin{proof}
$(1)$ We prove only the inclusion \( \subseteq \). Let \( M \in \mathsf{FP}_n\text{-}\mathsf{Inj}(R) \cap [\mathsf{FP}_n\text{-}\mathsf{Inj}(R)]^{\perp_1} \), and consider its injective envelope \( M \hookrightarrow E(M) \). Since \( R \) is left \( n \)-coherent, the quotient \( E(M)/M \) is also \( \mathsf{FP}_n \)-injective. Applying  \( \Hom_R(-, M) \) to the short exact sequence
$0 \to M \to E(M) \to E(M)/M \to 0,$
we obtain the exact sequence 
$\Hom_R(E(M), M) \to \Hom_R(M, M) \to 0.$ Hence, \( M \) is injective.\\
$(2)$ The proof is analogous to part (1). Note that if \( R \) is left \( n \)-coherent, then \( \mathsf{FP}_n\text{-}\mathsf{Proj}(R) \) is closed under kernels of epimorphisms.
\end{proof}

\begin{dfn}\label{def:regular}
Let \( R \) be a ring and let \( n \geq 0 \). We say that: 
\begin{itemize}
    \item \( R \) is \newterm{left \( n \)-regular} if every finitely \( n \)-presented \( R \)-module has finite projective dimension.
    \item \( R \) is \newterm{left \( n \)-von Neumann regular} if every finitely \( n \)-presented \( R \)-module is projective.
    \item \( R \) is \newterm{left uniformly \( n \)-regular} if there exists an integer \( k \geq 0 \) such that every finitely \( n \)-presented \( R \)-module has projective dimension at most \( k \).
\end{itemize}
\end{dfn}

The term $n$-von Neumann regular ring already appears in the literature; see \cite{Zhu}. However, it is used there to refer to commutative rings for which every finitely $n$-presented $R$-module is projective.

The relationships between these notions are summarized as follows:
\[
\begin{tikzcd}[column sep=small, row sep=small]
\text{$n$-von Neumann regular} \arrow[r, Rightarrow] 
& \text{uniformly $n$-regular} \arrow[r, Rightarrow] \arrow[d, Rightarrow] 
& \text{uniformly regular} \arrow[d, Rightarrow] \\
& \text{$n$-regular} \arrow[r, Rightarrow] 
& \text{regular}
\end{tikzcd}
\]

We make some elementary observations.

\begin{prop}\label{prop:regularity-hierarchy-unified}
Let \( R \) be a ring and let \( n \geq 0 \). Then:
\begin{enumerate}
    \item If \( R \) is left \( n \)-regular (respectively, uniformly \( n \)-regular), then it is left \( k \)-regular (respectively, uniformly \( k \)-regular) for all \( k > n \).
    \item If \( R \) is left \( n \)-regular (respectively, uniformly \( n \)-regular) for some \( n \geq 0 \), then it is left regular (respectively, uniformly regular).
    \item If \( R \) is left \( n \)-coherent, then the following are equivalent:
    \begin{enumerate}
        \item[(a)] \( R \) is left regular;
        \item[(b)] \( R \) is left  \( n \)-regular;
        \item[(c)] \( R \) is  left \( k \)-regular for some \( k > n \);
        \item[(d)] \( R \) is left \( k \)-regular for all \( k > n \).
    \end{enumerate}
\end{enumerate}
\end{prop}
\begin{proof}
Parts (1) and (2) follow from the inclusion \( \mathsf{FP}_{\infty}(R) \subseteq \mathsf{FP}_k(R) \subseteq \mathsf{FP}_n(R) \) for all \( k > n \).  
For part (3), if \( R \) is left \( n \)-coherent, then \( \mathsf{FP}_{\infty}(R) = \mathsf{FP}_k(R) = \mathsf{FP}_n(R) \) for all \( k \geq n \).
\end{proof}

As a consequence of the previous result and Proposition~\ref{prop: regula flat proj}, we have:

\begin{cor}\label{cor:regular glaz}
    Let \( R \) be a left \( n \)-coherent ring and let \( n \geq 0 \). Then:
    \begin{enumerate}
        \item \( R \) is left \( n \)-regular if and only if every \( R \)-module finitely $n$-presented  \( R \)-module has finite flat dimension.
        \item \( R \) is left uniformly \( n \)-regular if and only if there exists an integer \( k \geq 0 \) such that  
        $\mathsf{FP}_n(R) \subseteq \mathcal{F}_k(R).$  
        In the commutative local case, this is equivalent to  $
        \Tor_k^R(M, R/\mathsf{m}) = 0 \quad \text{for all } M \in \mathsf{FP}_n(R),$
        where \( \mathsf{m} \) is the maximal ideal of \( R \).
    \end{enumerate}
\end{cor}
\qed

We say that a ring \( R \) is \newterm{left \( n \)-coherent regular} if it is both left \( n \)-coherent and left regular. By Proposition \ref{prop:regularity-hierarchy-unified}, this is equivalent to being both left \( n \)-coherent and left \( n \)-regular. Rings with these properties were studied in \cite{ep}. By \cite[Theorem 2.1]{Zhu}, a ring \( R \) is left \( n \)-coherent if and only if every finitely \( (n{-}1) \)-presented submodule of a projective \( R \)-module is finitely \( n \)-presented. Combining this with Proposition \ref{prop:regular 0} and Proposition \ref{prop:regularity-hierarchy-unified}, we obtain the following result.

\begin{prop}\label{prop:regular 1}
Let \( R \) be a ring. Then \( R \) is left \( n \)-coherent regular if and only if every finitely \( (n{-}1) \)-presented submodule of a projective \( R \)-module is finitely \( n \)-presented and has finite projective dimension.
\end{prop}
\qed

Recall that if \( \mathcal{X} \) is a class of \( R \)-modules and $M$ is an \( R \)-module, the \newterm{relative projective dimension} of \( M \) with respect to \( \mathcal{X} \) is defined as
\[
\pd_{\mathcal{X}}(M) := \min\{ n \geq 0 : \Ext^j_R(M, -)|_{\mathcal{X}} = 0 \text{ for all } j > n \}.
\]
Dually, we denote by \( \id_{\mathcal{X}}(M) \) the \newterm{relative injective dimension} of \( M \) with respect to \( \mathcal{X} \). Furthermore, for any class \( \mathcal{Y} \subseteq \Mod R \), we define
\[
\pd_{\mathcal{X}}(\mathcal{Y}) := \sup\{ \pd_{\mathcal{X}}(Y) : Y \in \mathcal{Y} \} \quad \text{and} \quad \id_{\mathcal{X}}(\mathcal{Y}) := \sup\{ \id_{\mathcal{X}}(Y) : Y \in \mathcal{Y} \}.
\]
If \( \mathcal{X} = \Mod R \), we simply write \( \pd(\mathcal{Y}) \) and \( \id(\mathcal{Y}) \); see \cite{AM}.

\begin{prop} 
Let \( R \) be a left \( n \)-coherent ring. Then \( R \) is left \( n \)-regular if and only if every finitely \( n \)-presented \( R \)-module has finite projective dimension relative to the class \( \mathsf{FP}_n\text{-}\mathsf{Proj}(R) \).
\end{prop}

\begin{proof} 
Since \( (\mathsf{FP}_n\text{-}\mathsf{Proj}(R), \mathsf{FP}_n\text{-}\mathsf{Inj}(R)) \) forms a complete cotorsion pair by \cite[Theorem 2.2]{AM}, it follows that
\[
\pd_R(M) \;=\; \max\!\left\{ \pd_{\mathsf{FP}_n\text{-}\mathsf{Proj}(R)}(M),\ \pd_{\mathsf{FP}_n\text{-}\mathsf{Inj}(R)}(M) \right\}
\]
for every \( R \)-module \( M \).  

Now, if \( M \) is finitely \( n \)-presented, the assumption that \( R \) is \( n \)-coherent implies that 
$
\pd_{\mathsf{FP}_n\text{-}\mathsf{Inj}(R)}(M) = 0.$
Consequently, $$
\pd_R(M) \;=\; \pd_{\mathsf{FP}_n\text{-}\mathsf{Proj}(R)}(M).$$
\end{proof}

It is clear that \( R \) is left \( n \)-coherent regular if and only if every finitely \( n \)-presented \( R \)-module admits a finite resolution by finitely generated projective modules.  Recall that for any \( R \)-module \( M \), we denote by \( \mathsf{S}^n(M) \) the complex concentrated in degree \( n \), and by \( \mathsf{D}^n(M) \) the complex with \( M \xrightarrow{1_M} M \) in degrees \( n \) and \( n{-}1 \). A morphism of complexes is a quasi-isomorphism if and only if its mapping cone is acyclic. A complex is said to be \emph{perfect} if it is quasi-isomorphic to a bounded complex of finitely generated projective modules. Now, let \( M \) be a finitely \( n \)-presented module. Then \( \mathsf{S}^0(M) \) is quasi-isomorphic to a bounded complex \( \mathbf{P} \) of finitely generated projective modules. Since \( H_i(\mathbf{P}) = 0 \) for all \( i \neq 0 \) and \( M \cong H_0(\mathbf{P}) = Z_0\mathbf{P}/B_0\mathbf{P} \), we obtain a finite resolution
$$0 \longrightarrow P_n \longrightarrow \cdots \longrightarrow P_1 \longrightarrow Z_0\mathbf{P} \longrightarrow M \longrightarrow 0,$$
with each term finitely generated and projective.

Conversely, if every bounded complex of finitely \( n \)-presented modules is perfect, then in particular, each \(  \mathsf{S}^0(M) \), for \( M \in \mathsf{FP}_n(R) \), is quasi-isomorphic to a bounded complex of finitely generated projectives. Following the same argument as in the case \( n = 1 \); see \cite[Theorem 2.2]{GI}, we deduce that \( M \) admits a finite projective resolution. Therefore, we obtain the following characterization:

\begin{theorem}
Let \( R \) be a ring. The following statements are equivalent:
\begin{enumerate}
    \item \( R \) is left $n$-coherent regular.
    \item Every finitely $n$-presented \( R \)-module admits a finite resolution by finitely generated projective modules.
    \item Every bounded chain complex of finitely $n$-presented \( R \)-modules is perfect.
\end{enumerate}
\end{theorem}
\qed

Next, we explore the connection between \( n \)-coherence and regularity in extensions of rings.  The following lemma establishes a well-known relationship between finitely \( n \)-presented modules and faithfully flat ring extensions. Its proof is straightforward.

\begin{lemma}\label{lemm:faithfully}
Let $\varphi: R \to S$ be a ring homomorphism, and let $M$ be a right $R$-module.
\begin{enumerate}
    \item If $S$ is flat as a left $R$-module and $M$ is a finitely $n$-presented right $R$-module, then $M \otimes_R S$ is a finitely $n$-presented right $S$-module.
    \item If $S$ is faithfully flat as a left $R$-module and $M \otimes_R S$ is a finitely $n$-presented right $S$-module, then $M$ is a finitely $n$-presented right $R$-module.
\end{enumerate}
\end{lemma}
\qed

  \begin{prop}
Let $\varphi: R \to S$ be a ring homomorphism, and assume that $S$ is faithfully flat as a left $R$-module.
\begin{enumerate}
    \item If $S$ is a right $n$-coherent regular ring, then $R$ is also a right $n$-coherent regular ring.
    \item If $S$ is a right $n$-coherent uniformly regular ring, then $R$ is also a right $n$-coherent uniformly regular ring.
\end{enumerate}
\end{prop}
\begin{proof}    
\begin{enumerate}
    \item Assume that \( S \) is a right \( n \)-coherent regular ring. Then, by Lemma \ref{lemm:faithfully}, \( R \) is also right \( n \)-coherent regular. Let \( M \) be a finitely \( n \)-presented right \( R \)-module. This means there exists an exact sequence:
    $$\cdots \xrightarrow{d_{n+1}} F_n \xrightarrow{d_n} F_{n-1} \xrightarrow{d_{n-1}} \cdots \xrightarrow{d_1} F_0 \xrightarrow{u_0} M \to 0,$$
    where each \( F_i \) is a finitely generated free right \( R \)-module. Tensoring this sequence with \( S \) yields the following presentation for \( M \otimes_R S \):
    $$\cdots \xrightarrow{d_{n+1} \otimes 1_S} F_n \otimes_R S \xrightarrow{d_n \otimes 1_S} F_{n-1} \otimes_R S \to \cdots \to M \otimes_R S \to 0.$$
    It follows that \( \operatorname{pd}_S(M \otimes_R S) = m < \infty \), and hence:
    $$K_{m-1} \otimes_R S \cong \ker(d_{m-1} \otimes 1_S),$$
    where \( K_{m-1} = \ker(d_{m-1}) \) is finitely generated and projective.  

    Since \( S \) is faithfully flat over \( R \), it follows that \( K_{m-1} \) is a projective \( R \)-module. Thus, \( \operatorname{pd}_R(M) < \infty \), completing the argument.  
    \item The proof is similar to part (1).
\end{enumerate}
\end{proof}

Recall that if \( R \) and \( S \) are two rings, then for any \( T \)-module \( M \), where \( T = R \oplus S \), there exists a unique decomposition \( M = M_1 \oplus M_2 \). Here, \( M_1 = (R, 0)M \) is an \( R \)-module, and \( M_2 = (0, S)M \) is an \( S \)-module. For \( x \in M_1 \) and \( r \in R \), the action is given by \( rx = (r, 0)x \), and for \( y \in M_2 \) and \( s \in S \), the action is \( sy = (0, s)y \). Following \cite[Lemma 6.4]{Li}, the module \( M = M_1 \oplus M_2 \in  \mathsf{FP}_n(T) \) if and only if \( M_1 \in \mathsf{FP}_n(R) \) and \( M_2 \in  \mathsf{FP}_n(S) \). Furthermore, \( \pd_T(M) < \infty \) if and only if \( \pd_R(M_1) < \infty \) and \( \pd_S(M_2) < \infty \).  The following theorem, which follows immediately, shows that the class of \( n \)-coherent regular rings is closed under finite direct products.

\begin{theorem}\label{theorem:finite direct products}
 Let \( R \) and \( S \) be rings. Then \( T = R \oplus S \) is a left \( n \)-coherent regular ring if and only if both \( R \) and \( S \) are left \( n \)-coherent regular rings.
\end{theorem}
\qed
   
Theorem \ref{theorem:finite direct products} gives more examples of $n$-coherent regular rings. \\

\subsection{Commutative case}
The notion of \newterm{Gorenstein dimension} (or \newterm{G-dimension}) for finitely generated modules over a Noetherian ring was introduced by Auslander and Bridger; see \cite{AB69}, and later extended by Enochs and Jenda to arbitrary rings via the \newterm{Gorenstein projective dimension}; see \cite{EJ95}.  Given a commutative ring \( R \), a finitely generated \( R \)-module \( M \) belongs to the \(\mathsf{G}\)-class \( \mathsf{G}(R) \) if  
$$\Ext_R^i(M, R) = 0 \quad \text{and} \quad \Ext_R^i(M^*, R) = 0 \quad \text{for all } i \geq 1,$$
and the canonical map \( M \to M^{**} \) is an isomorphism.  
The \newterm{restricted} \(\mathsf{G}\)-class, denoted \( \widetilde{\mathsf{G}}(R) \), consists of modules in \( \mathsf{G}(R) \) that are of type \(\mathsf{FP}_{\infty}(R)\), together with their duals. A complex \( \mathbf{G} \) is called a \newterm{G-resolution} (resp. \(\widetilde{G}\)-resolution) of \( M \) if each \( G_i \) is a member of the \(\mathsf{G}\)-class (respectively, \( \widetilde{\mathsf{G}} \)-class), \( G_n = 0 \) for \( n < 0 \), \( H_i(\mathbf{G}) = 0 \) for \( i \neq 0 \) and \( H_0(\mathbf{G}) \cong M \). For a resolution \( \mathbf{G} \), we define its \newterm{length} as the infimum of the integers \( n \) such that \( G_n \neq 0 \).  
If \( M \) admits a \( G \)-resolution (resp. \( \widetilde{G} \)-resolution), the \newterm{G-dimension} \( \mathsf{Gdim}_R M \) (resp. \( \mathsf{\widetilde{G}dim}_{R} M \)) is the infimum of the lengths of all such resolutions of \( M \). If no such resolution exists, the corresponding dimension is said to be \newterm{undefined}. By \cite[Proposition 3.8]{HM}, if \( M \) is a finitely \( \infty \)-presented \( R \)-module with finite projective dimension, then its restricted Gorenstein dimension coincides with its projective dimension, i.e., $\mathsf{\widetilde{G}dim}_{R}(M) = \pd_R(M).$ Then by \cite[Corollary 3.6]{HM} we have:

\begin{prop}
Let \( R \) be a commutative \( n \)-coherent regular ring with \( n \geq 0 \). 
Then for every nonzero finitely \( n \)-presented \( R \)-module \( M \), 
$$\pd_R(M) \;=\; \sup \{\, i \geq 0 \mid \Ext_R^i(M, R) \neq 0 \,\}.$$
In particular, a nonzero finitely presented \( R \)-module \( M \) is projective 
if and only if 
\( R \in {\{M\}}^{\perp_{\infty}} \).
\end{prop}
\qed

Recall that if $
\cdots \longrightarrow P_i \xrightarrow{f_i} P_{i-1} \longrightarrow \cdots 
\longrightarrow P_1 \xrightarrow{f_1} P_0 \xrightarrow{f_0} M \longrightarrow 0 $
is a projective resolution of \( M \), then the module \(\operatorname{Im}(f_i)\) is called the 
\(i\)-th \newterm{syzygy} of \( M \). We denote it by \(\Omega^i(M)\), with the convention 
\(\Omega^0(M) = M\). For such a resolution we set 
$$S(M) \;=\; \bigoplus_{i \geq 0} \Omega^i(M).$$
Moreover, 
$$\{M\}^{\perp_{\infty}} \;=\; \{S(M)\}^{\perp_1}.$$

\begin{cor}
Let \( R \) be a commutative \( n \)-coherent regular ring with \( n \geq 0 \). 
Then a nonzero finitely presented \( R \)-module \( M \) is projective 
if and only if 
\[
R \in \{S(M)\}^{\perp_{1}}.
\]
\end{cor}
\qed

\begin{prop}\label{prop: pair projectivo}
Let \( R \) be a commutative \( n \)-coherent regular ring with \( n \geq 0 \). Then,  
$$\mathsf{FP}_n\text{-}\mathsf{Inj}(R) \bigcap \big(\mathsf{FP}_n(R) \setminus \mathsf{proj}(R)\big) = \emptyset.$$
\end{prop}
\begin{proof}
Suppose there exists a non-projective \( R \)-module \( M \in \mathsf{FP}_n(R) \) such that \( M \in \mathsf{FP}_n\text{-}\mathsf{Inj}(R) \). Then \( \Ext_R^1(M, M) = 0 \), and consequently \( \Ext_R^k(M, M) = 0 \) for all \( k \geq 2 \). Since \( M \) is finitely \( n \)-presented and non-projective, we have \( 1 \leq \pd_R(M) = s < \infty \). However, by \cite[Lemma 1.2]{PT}, this implies \( \Ext_R^s(M, M) \neq 0 \), a contradiction.
\end{proof}

The following result is an immediate consequence of the previous proposition, together with the fact that in the \(n\)-coherent case the inclusion $\varinjlim \mathsf{FP}_n\text{-}\mathsf{Inj}(R) \subseteq \mathsf{FP}_n\text{-}\mathsf{Inj}(R)$
holds, see \cite[Theorem 4.1]{MD4}. 

\begin{cor} 
Let \( R \) be a commutative \( n \)-coherent regular ring with \( n \geq 0 \). Then, 
$$\varinjlim \mathsf{FP}_n\text{-}\mathsf{Inj}(R) \bigcap \big(\mathsf{FP}_n(R) \setminus \mathsf{proj}(R)\big) = \emptyset.$$
\end{cor}
\qed

We now consider properties of the trivial extension \( R := A \ltimes M \).  Recall that if \( I \subseteq A \) is an ideal and \( M' \subseteq M \) is a submodule such that \( IM \subseteq M' \), 
then the subset \( J := I \ltimes M' \subseteq R \) is an ideal of \( R \), though not all ideals of \( R \) are of this form. However, prime (resp. maximal) ideals of \( R \) have the form \( P \ltimes M \), where \( P \) is a prime (resp. maximal) ideal of \( A \).  The ring \( R \) is isomorphic to the subring
$$\Gamma = \left\{
\begin{pmatrix}
a & m \\
0 & a
\end{pmatrix}
\;\middle|\;
a \in A,\, m \in M
\right\}
\subseteq
\begin{pmatrix}
A & M \\
0 & A
\end{pmatrix}$$
via the correspondence
$$(a, m) \mapsto
\begin{pmatrix}
a & m \\
0 & a
\end{pmatrix}.$$

\begin{prop}
Let \( A \) be a domain that is not a field, \( K = \mathsf{Q}(A) \), and \( R = A \ltimes K \). For \( n \ge 2 \), \( R \) is \( n \)-coherent regular if and only if \( A \) is.  
\end{prop}

\begin{proof}
By \cite[Theorem 3.1]{KMa}, \( R \) is \( n \)-coherent if and only if \( A \) is. By Proposition \ref{prop:regular 1}, an \( n \)-coherent ring is regular  if and only if every finitely \( (n-1) \)-presented submodule of a projective module has finite projective dimension. 
The result follows from \cite[Lemma 3.3]{KMa} and the faithful flatness of $R$ over $A$.
\end{proof}

Recall that a ring \( R \) is called \newterm{arithmetic} if it is locally a valuation ring. A \newterm{Bézout ring} (also known as an \newterm{\( \mathsf{F}\)-ring}) is one in which every finitely generated ideal is principal. The ring \( R \) is called \newterm{Hermite} if for every pair \( (a,b) \in R^2 \), there exist elements \( d, a', b' \in R \) such that \( a = d a' \), \( b = d b' \), and \( R a' + R b' = R \). Furthermore, \( R \) is a \newterm{Kaplansky ring} (or \newterm{elementary divisor ring}) if for every matrix \( A \) over \( R \), there exist invertible matrices \( P, Q \) and a diagonal matrix \( D \) such that $
P A Q = D.$ These concepts are related by the following implications that are generally not reversible.
$$
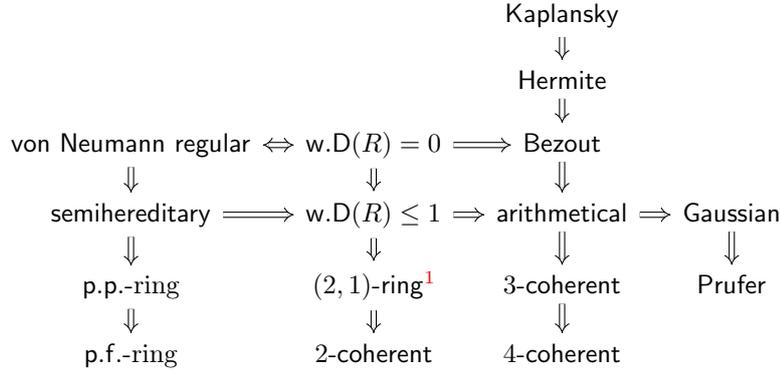
\begin{tikzcd}[row sep=small, column sep=small]
     & & \mathsf{Kaplansky} \arrow[d, Rightarrow] & & &\\
     & &\mathsf{Hermite} \arrow[d, Rightarrow]  &  & &\\
  \mathsf{von\text{ }Neumann\text{ }regular} \arrow[r, Leftrightarrow]  \arrow[d, Rightarrow]  &\wD(R)=0 \arrow[r, Rightarrow] \arrow[d, Rightarrow] &\mathsf{Bezout} \arrow[d, Rightarrow]  & & & \\
 \mathsf{semihereditary} \arrow[r, Rightarrow] \arrow[d, Rightarrow]
    & \wD(R)\leq 1 \arrow[r, Rightarrow] \arrow[d, Rightarrow]
    & \mathsf{arithmetical} \arrow[r, Rightarrow] \arrow[d, Rightarrow] & \mathsf{Gaussian} \arrow[d, Rightarrow]
    & & \\
    \mathsf{p.p.}\text{-ring} \arrow[d, Rightarrow] 
    & (2,1)\text{-}\mathsf{ring}\footnote{Rings in which every finitely 2-presented module has projective dimension at most one}
\arrow[d, Rightarrow] 
    & 3\text{-}\mathsf{coherent}  \arrow[d, Rightarrow] 
    & \mathsf{Prufer} & &\\
   \mathsf{p.f.}\text{-ring}
    & 2\text{-}\mathsf{coherent} 
    & 4\text{-}\mathsf{coherent} 
    & & &\\
\end{tikzcd}.$$

\begin{ex}  Let \( \Ical \) be a family of pairwise disjoint intervals of the real line with rational endpoints, such that between any two intervals of \( \Ical \), there is at least one other interval of \( \Ical \). Define \( R \) as the ring of continuous maps \( \mathbb{R} \to \mathbb{R} \) that are rational constants on each interval of \( \Ical \), except on finitely many intervals of \( \Ical \), where they are given by a rational polynomial. It follows from \cite[Example 3.3]{Cou1} that \( R \) is arithmetic and reduced (hence \( \wD(R) = 1 \)). Moreover, by \cite[Theorem II.1]{Cou2}, \( R \) is \( 2 \)-coherent. Therefore, \( R \) is a \( 2 \)-coherent regular ring. Finally, note that \( R \) is not coherent.  
\end{ex}

We recall from Proposition \ref{prop:local regula} that local coherent regular rings satisfy \(\mathsf{m}^2 \neq 0\).  In contrast, local rings with \(\mathsf{m}^2 = 0\) may still be \(2\)-coherent, as established in the following proposition.

\begin{prop}\label{prop:local 2 coh}
      Let \((R,\mathsf{m})\) be a local commutative ring with \(\mathsf{m}^2 = 0\). Then \(R\) is a \(2\)-coherent ring. 
\end{prop}
\begin{proof}
    Follows from \cite[Lemma 4.6]{AJK11} and \cite[Theorem 6]{C17}.
\end{proof}

\begin{ex}
Let \(\mathsf{k}\) be a field and let \(\{x_1, x_2, \ldots\}\) be an infinite set of indeterminates over \(\mathsf{k}\). 
Consider the ring 
$$R \;=\; \frac{\mathsf{k}[x_1, x_2, \ldots]}{\mathsf{m}^2}
   \;\cong\; \mathsf{k}[\overline{x}_1,\overline{x}_2,\ldots],$$
where \(\mathsf{m} := (x_1, x_2, \ldots)\). 
Then \(R = \mathsf{k} \,+\, \frac{\mathsf{m}}{\mathsf{m}^2}\) is a local ring with maximal ideal \(\frac{\mathsf{m}}{\mathsf{m}^2}\). 
By Proposition \ref{prop:local 2 coh}, it is \(2\)-coherent. 
However, by \cite[Example 3.13]{AJK11}, \(R\) is not coherent. 
\end{ex}

\begin{cor}
    Let \((R, \mathsf{m})\) be a local commutative coherent ring such that \(\mathsf{m}^2 = 0\). Then \(\wD(R)\) is $\infty$. 
\end{cor}
\begin{proof}
Since \(\mathsf{m}^2 = 0\), \(R\) is not regular by Proposition \ref{prop:local regula}. The claim follows from \cite[Theorem 3.11]{AJK11}.
\end{proof}

    \begin{prop}
Let \( R \) be a commutative \( n \)-coherent regular ring with \( n \geq 2 \). If \( r \in R \) is nonzero and \( r^2 = 0 \), then the ideal \( I := (r) \) satisfies \( I \subsetneq \mathsf{Ann}(I) \).
\end{prop}
  \begin{proof}  
    Let \( I = (r) \), where \( r \in R \setminus \{0\} \) and \( r^2 = 0 \).  Clearly, \( I \subseteq \mathsf{Ann}(I) \). If equality holds, then \( I \) is finitely \( n \)-presented with \( \pd_R(I) = \infty \) by \cite[Lemma 4.5]{Wang}, contradicting the regularity of \( R \). Thus, \( I \subsetneq \mathsf{Ann}(I) \).
\end{proof} 

\begin{prop}
Let \( R \) be a commutative \( n \)-coherent regular ring. Then every \( \mathsf{FP}_n \)-flat \( R \)-module is torsion-free.
\end{prop}
\begin{proof}
    Follows from \cite[Proposition 4.2]{WZKXS}.
\end{proof}

Thus, for commutative \( n \)-coherent regular rings, we obtain the following ascending chain of inclusions:  
$$\mathsf{Flat}(R) = \mathsf{FP}_0\text{-}\mathsf{Flat}(R) = \mathsf{FP}_1\text{-}\mathsf{Flat}(R) \subseteq  \cdots \subseteq \mathsf{FP}_{n}\text{-}\mathsf{Flat}(R) \subseteq \mathsf{TFree}(R).$$

If \( R \) is an \( n \)-coherent regular domain, then the equality \( \mathsf{FP}_{n}\text{-}\mathsf{Flat}(R) = \mathsf{TFree}(R) \) holds if and only if \( \sfindim(R) \leq 1 \); see \cite[Theorem 4.13]{WZKXS}. 

\section{Vogel-Regularity}\label{sec:Vogel regularidad}  
A more general notion of regularity for rings and exact categories was introduced by Vogel; see \cite{Vogel}. Let \( \mathcal{C} \) be a class in \( \lMod R \). We say that the class \( \mathcal{C} \) is \newterm{Vogel} if it satisfies the following conditions:
\begin{enumerate}
    \item \( \mathcal{C} \) is closed under direct limits.
    \item \( \mathcal{C} \) has the 2-out-of-3 property with respect to \( \lMod R \).
    \item \( \mathcal{C} \) contains all finitely generated projective \( R \)-modules.
\end{enumerate}

If \(\mathcal{C}\) is a Vogel class, then it is also closed under direct summands; see \cite[Remark 2]{PPT}. Moreover, if \(\mathcal{C}'\) denotes the class of all direct summands of modules in \(\mathcal{C}\), we have $
\mathcal{C} = \varinjlim \mathcal{C} = \varinjlim \mathcal{C}'.$

\begin{dfn}
Let \(R\) be a ring. We say that \(R\) is left \newterm{Vogel-regular} if the only Vogel class in \(\lMod R\) is \(\lMod R\) itself.
\end{dfn}

\begin{prop}  
Let \(R\) be a left \(n\)-coherent regular ring with \(n \ge 0\), and let \(\mathcal{C}\) be a Vogel class of \(R\)-modules.  Then \( \mathsf{FP}_n(R) \subseteq \mathcal{C} \).
\end{prop}  

\begin{proof}  
Every finitely \( n \)-presented \( R \)-module \( M \) has a finite projective resolution  
$$0 \longrightarrow P_k \longrightarrow \dots \longrightarrow P_1 \longrightarrow P_0 \longrightarrow M \longrightarrow 0,$$  
with each \( P_i \) finitely generated and projective.  
Since \( \mathcal{C} \) contains all finitely generated projective modules and satisfies the \( 2 \)-out-of-\( 3 \) property, splitting the resolution into short exact sequences implies that \( M \in \mathcal{C} \).
\end{proof}  

In what follows, right $R$-modules are identified with left $R^{\op}$-modules.

\begin{theorem}
    Let \( R \) be a left Vogel-regular ring. Then, for some \( n \geq 2 \), the class $
    [\mathsf{FP}_n\text{-}\mathsf{Flat}(R^{\op})]^{\top_1}$
    satisfies the 2-out-of-3 property in \(\lMod R\) if and only if \( R \) is left coherent.
\end{theorem}

\begin{proof}
    Suppose that the class $[\mathsf{FP}_n\text{-}\mathsf{Flat}(R^{\op})]^{\top_1}$ satisfies the 2-out-of-3 property for some \( n \geq 2 \). By \cite[Lemma 2.1]{PPT}, the class
    $$\mathcal{L} := \varinjlim \mathsf{FP}_n(R^{\op})$$
    is closed under direct limits and coincides with $[\mathsf{FP}_n\text{-}\mathsf{Flat}(R^{\op})]^{\top_1}$. Since \(\mathcal{L}\) contains all projective left \( R \)-modules, it follows that \(\mathcal{L}\) is a Vogel class. By the Vogel-regularity of \( R \), this implies \(\mathcal{L} = \lMod R\). In particular, every  \( R \)-module is a direct limit of finitely \( n \)-presented modules. Then, by \cite[Corollary 5.3]{Li}, \( R \) is left coherent. The converse is straightforward.
\end{proof}

Over coherent rings, Vogel-regularity coincides with classical regularity \cite[Theorem 1]{Vogel}.

\begin{cor}
    Let \( R \) be a ring. Then \( R \) is left coherent regular if and only if \( R \) is left Vogel-regular and the class $[\mathsf{FP}_n\text{-}\mathsf{Flat}(R^{\op})]^{\top_1}$
    satisfies the 2-out-of-3 property in \(\lMod R\) for some \( n \geq 2 \).
\end{cor}
\qed

\section{Uniform $n$-regularity}\label{sec: uniformly regu}
The study of module classes with projective dimension bounded  by some positive integer provides large insight to both the module category over the ring, as well as intrinsic properties of the ring itself. Let \( n \) and \( d \) be nonnegative integers. A ring \( R \) is called a \newterm{left \((n, d)\)-ring} if every finitely \( n \)-presented \( R \)-module has projective dimension \(\leq d\). Similarly, \( R \) is called a \newterm{left weak \((n, d)\)-ring} if every finitely \( n \)-presented cyclic  \( R \)-module has projective dimension \(\leq d\), or equivalently, if every finitely \((n-1)\)-presented left ideal of \( R \) has projective dimension \(\leq d - 1\). It follows immediately from this definition that a ring \( R \) is left uniformly \( n \)-regular if and only if \( R \) is a left \((n, d)\)-ring for some \( d \geq 0 \).  The notion of uniform $n$-regularity was introduced due to the poor behavior of regularity with respect to infinite products.

Clearly, there exist uniformly \(n\)-regular rings (for instance, rings of finite global dimension) that are not \(1\)-coherent. Nevertheless, the following proposition shows that uniform \(n\)-regularity implies a \(k\)-coherence condition on the ring for some \(k \geq n\). This fact is well known in the commutative setting \cite[Theorem 2.2]{Costa}.

\begin{prop}\label{prop: sup} 
Let \( R \) be a left uniformly \( n \)-regular ring. Then \( R \) is left \( k \)-coherent for some \( k \geq n \).  
\end{prop}

\begin{proof}
Suppose that \( R \) is left uniformly \( n \)-regular. Then there exists an integer \( d \) such that every  \( R \)-module in \( \mathsf{FP}_n(R) \) has projective dimension at most \( d \).  
Set \( k := \max\{n, d\} \). We claim that \( R \) is left \( k \)-coherent.  Let 
$P_k \xrightarrow{u_k} P_{k-1} \to \cdots \to P_0 \to M \to 0$
be a finite \( k \)-presentation, i.e., \( M \in \mathsf{FP}_k(R) \).  
Since \( k \geq n \), we have \( M \in \mathsf{FP}_n(R) \), hence \(\pd_R(M) \leq d\).  
In particular, \(\ker u_{d-1}\) is projective, and because \(d-1 < k\), it follows that \(\ker u_{d-1} = \mathrm{Im}\, u_d\) is finitely generated.  Therefore \( M \in \mathsf{FP}_\infty(R) \).
\end{proof}

Although left uniformly \( n \)-regular rings are left \( k \)-coherent for some \( k \geq n \), this implication does not extend to uniformly regular rings, as shown in the following example.

\begin{ex}
A uniformly regular ring that is not \( n \)-coherent for any \( n \geq 0 \):
Let \( d \geq 0 \) be a fixed integer and let \( k \geq 2 \). Consider the ring \( R_k \) defined as the quotient of the path algebra of \( \mathsf{Q}_{k,d} \) over a field \( \mathsf{k} \), modulo the ideal generated by all paths of length \( \ell \geq 2 \), see \cite[Theorem 2.1]{Liliu}. Here, \( \mathsf{Q}_{k,d} \) is the quiver given by:
\begin{itemize}
    \item \( \mathsf{Q}_{k,d}\) has \( k + d + 1 \) vertices.  
    \item For each \( i \in \{0, 1, \dots, k + d - 1\} \setminus \{d\} \), there is a single arrow \( \alpha_{i+1} \) from vertex \( i+1 \) to vertex \( i \).  
    \item There are infinitely many arrows $\{\beta_s: s\in \mathsf{S} \}$ from vertex \( d+1 \) to vertex \( d \).
    \item There are infinitely many arrows $\{\gamma_s: s\in \mathsf{S} \}$ from vertex \( d \) to vertex \( d+1 \).
\end{itemize}
$$\begin{tikzcd}[row sep=large, column sep=small, ampersand replacement=\&]
\cdots \arrow[r, "\alpha_{k+d+1}"'] \&
\bullet_{k+d} \arrow[r, ""'] \&
\cdots \arrow[r, "\alpha_{d+2}"'] \&
\bullet_{d+1} 
  \arrow[rr, bend left=50, "\alpha_{\beta_s}"] 
  \&
  \&
\bullet_d 
  \arrow[ll, bend left=50, "\alpha_{\gamma_s}"'] 
  \arrow[r, "\alpha_d"'] \&
\cdots \arrow[r, "\alpha_2"'] \&
\bullet_1 \arrow[r, "\alpha_1"'] \&
\bullet_0
\end{tikzcd}$$
Let \( R = \prod_{k=2}^{\infty} R_k \). By \cite[Example 4.8]{Li}, we have $\sup\{\pd_R(M) \mid M \in \mathsf{FP}_\infty(R)\} \leq d.$
Thus, \( R \) is right uniformly regular ring. However, by \cite[Corollary 2.2]{Liliu}, \( R \) is not right \( n \)-coherent for any \( n \geq 0 \).  
\end{ex}

\begin{prop}\label{prop: debil unif}
Let \( R \) be a ring such that \( \fd_R(M) \leq d \) for every \( M \in \mathsf{FP}_k(R) \). 
Then \( R \) is left uniformly \( n \)-regular, where \( n = \max\{k,\, d+1\} \). In particular, $R$ is left  $n$-coherent.
\end{prop}

\begin{proof}
By \cite[Corollary 5.8]{Liliu}, \( R \) is left \( n \)-coherent with \( n = \max\{k,\, d+1\} \). 
Hence \(\mathsf{FP}_n(R) = \mathsf{FP}_\infty(R)\). 
Therefore, for every \( M \in \mathsf{FP}_n(R) \) we have \(\pd_R(M) = \fd_R(M) \leq d\).
\end{proof}

It is well-known that \(\wD(R[x]) = \wD(R) + 1 \leq k + 1\). Hence, by Proposition \ref{prop: debil unif} we have:

\begin{cor}
    Let \( R \) be a ring such that \( \wD(R) \leq k \). Then \( R[x] \) is left \((k+2)\)-coherent and uniformly \((k+2)\)-regular.
\end{cor}
\qed

The class of \((n,d)\)-rings  provide a framework for exploring the structure of non-Noetherian rings. If \( n \leq n' \) and \( d \leq d' \), then every left \((n, d)\)-ring is also a left \((n', d')\)-ring. In the commutative case, these rings were extensively studied by Costa \cite{Costa}. In particular, \cite[Theorem 1.3]{Costa} provides several fundamental characterizations for commutative rings:

\begin{enumerate}
    \item[(i)] \(R\) is a \((0,0)\)-ring if and only if \(R\) is a finite direct sum of fields.
    \item[(ii)] \(R\) is a \((0,1)\)-ring if and only if \(R\) is hereditary.
    \item[(iii)] \(R\) is a \((0, d)\)-ring if and onnly if \(\mathrm{gl. dim.}(R) \leq d\).
    \item[(iv)] \(R\) is a \((1,0)\)-ring if and only if \(R\) is von Neumann regular.
    \item[(v)] \(R\) is a \((1,1)\)-ring if and only if \(R\) is semihereditary.
    \item[(vi)] For all \(n \geq 0\), \(R\) is an \((n,0)\)-domain if and only if \(R\) is a field.
    \item[(vii)] \(R\) is a \((0,1)\)-domain if and only if \(R\) is a Dedekind domain.
    \item[(viii)] \(R\) is a \((1,1)\)-domain if and only if \(R\) is a Prüfer domain.
    \item[(ix)] If \(R\) is Noetherian, then \(R\) is an \((n, d)\)-ring if and only if \(\mathrm{gl. dim.}(R) \leq d\).
\end{enumerate}

 In the particular cases \( d = 0 \) or \( d = 1 \), we obtain the classes of left \( n \)-von Neumann regular rings and left \( n \)-hereditary rings, respectively.  

\begin{prop}
 Let \( R \) be a left \( n \)-regular ring with \( n \geq 1 \). Then $R$ is left \( n \)-hereditary ring if and only if every finitely \( (n-1) \)-presented submodule of a projective \( R \)-module with finite projective dimension is projective.   
  \end{prop}
\begin{proof}
Assume \( R \) is left \( n \)-regular and that every finitely \( (n-1) \)-presented submodule of a projective module with finite projective dimension is projective. Let \( M \) be a finitely \( (n-1) \)-presented submodule of a projective. Then, there exists a finitely generated free module $F$ such that $M$ is a submodule of $F$. Since \( R \) is left \( n \)-regular, we have \( \pd_R(F/M) < \infty \), and hence \( \pd_R(M) < \infty \). By hypothesis, \( M \) is projective. Therefore, by \cite[Theorem 3.2]{Zhu}, \( R \) is left \( n \)-hereditary. The converse is clear.
\end{proof}

\begin{theorem}\label{teo: n-hereditario}
    Let \( R \) be  ring and $n\geq 1$. The following conditions are equivalent:
    \begin{enumerate}
        \item \( R \) is a left \( n \)-hereditary ring.
        \item Homomorphic images of \(\mathsf{FP}_n \)-injective \( R \)-modules are \(\mathsf{FP}_n \)-injective.
        \item Every quotient of an injective module is \(\mathsf{FP}_n \)-injective.
        \item The sum of any two \( \mathsf{FP}_n \)-injective submodules of an \( R \)-module is \(\mathsf{FP}_n \)-injective.
        \item The sum of any two injective submodules of an \( R \)-module is \(\mathsf{FP}_n \)-injective.       
    \end{enumerate}
\end{theorem}
\begin{proof}
    The result follows by taking \( \mathcal{L} = \mathsf{FP}_n(R) \) in \cite[Theorem 3.3]{Zhou3}.
\end{proof}

\begin{cor}
Let $R$ be a left $n$-coherent ring with $n \geq 1$.  
If the intersection of any two injective submodules of an $R$-module is $\mathsf{FP}_n$-injective, then $R$ is left $n$-hereditary.
\end{cor}
\begin{proof}
Assume that $R$ is left $n$-coherent and that the intersection of any two injective submodules of an $R$-module is $\mathsf{FP}_n$-injective.  
Let $A$ and $B$ be injective submodules of some $R$-module.  
Consider the canonical short exact sequence
$$0 \longrightarrow A \cap B \longrightarrow A \oplus B \longrightarrow A + B \longrightarrow 0.$$
Since $A \oplus B$ is injective, and the class of $\mathsf{FP}_n$-injective modules is closed under cokernels of monomorphisms whenever $R$ is left $n$-coherent, it follows that $A+B$ is $\mathsf{FP}_n$-injective.  Therefore, by Theorem~\ref{teo: n-hereditario}, $R$ is left $n$-hereditary.
\end{proof}

If \( R \) is a left \( n \)-von Neumann regular ring, then \( \mathsf{FP}_n(R) = \mathsf{proj}(R) \); hence \( R \) is left \( n \)-coherent.  In particular, a \( 1 \)-von Neumann regular ring coincides with the classical notion of a von Neumann regular ring.  Moreover, it is clear that every left $n$-von Neumann regular ring is left self-$\mathsf{FP}_n$-injective.

\begin{prop}\label{prop: coherencia self fp}
Let \( R \) be a left self \( \mathsf{FP}_n \)-injective ring with \( n \geq 1 \). The following conditions are equivalent:
\begin{enumerate}
    \item \( R \) is left \( n \)-von Neumann regular.
    \item \( R \) is left \( n \)-hereditary.
    \item  \( R \) is a left \((n,d)\)-ring with \( 1\leq d\leq n\).
    \item \( R \) is left \( n \)-coherent regular.
\end{enumerate}
\end{prop}
\begin{proof} The implications \((1) \Rightarrow (2) \Rightarrow (3)  \Rightarrow (4)\) are clear.  
    $(3)\Rightarrow (1)$ Suppose \( R \) is left \( n \)-coherent regular. To prove that \( R \) is left \( n \)-von Neumann regular, it suffices to show that every finitely \( n \)-presented \( R \)-module is projective. If some \( M \in \mathsf{FP}_n(R) \) is non-projective, then \( 0 < \pd_R(M) < \infty \), contradicting \cite[Corollary 4.6]{gp2}. Thus, \( \mathsf{FP}_n(R) \subseteq \mathsf{proj}(R) \). 
\end{proof}

Recall that a cotorsion pair  \( (\mathcal{A}, \mathcal{B}) \) in \( \lMod R \) is called \newterm{projective} if it is hereditary, complete, and $\mathcal{A} \cap \mathcal{B} = \mathsf{Proj}(R).$ 

\begin{cor}\label{cor: von}
Let \( R \) be a left \( n \)-coherent regular ring with $n\geq 1$. Then the cotorsion pair 
$[\mathsf{FP}_n\text{-}\mathsf{Proj}(R), \ \mathsf{FP}_n\text{-}\mathsf{Inj}(R)]$
is projective in \( \lMod R \) if and only if \( R \) is left \( n \)-von Neumann regular.
\end{cor}
\begin{proof}
Suppose that $[\mathsf{FP}_n\text{-}\mathsf{Proj}(R), \ \mathsf{FP}_n\text{-}\mathsf{Inj}(R)]$ is projective. Then \( R \) is self \( \mathsf{FP}_n \)-injective. By Proposition \ref{prop: coherencia self fp}, it follows that \( R \) is left \( n \)-von Neumann regular. Conversely, assume that \( R \) is left \( n \)-von Neumann regular. Then every \( R \)-module is \( \mathsf{FP}_n \)-injective, and consequently every \( \mathsf{FP}_n \)-projective module is projective. Therefore,
$\mathsf{FP}_n\text{-}\mathsf{Proj}(R) \cap \mathsf{FP}_n\text{-}\mathsf{Inj}(R) \;=\; \mathsf{Proj}(R).$
\end{proof}

In particular, under the assumptions of Corollary \ref{cor: von}, this cotorsion pair coincides with the canonical projective cotorsion pair $(\mathsf{Proj}(R), \lMod R)$.

Equivalently, a cotorsion pair \( (\mathcal{A}, \mathcal{B}) \) is projective if and only if it is complete, and the class \(\mathcal{B}\)  contains all projective modules, and  satisfies the 2-out-of-3 property in \( \lMod R\). In this context we obtain the following proposition. 

\begin{prop}
Let \( R \) be a ring and let \( n \geq 1 \). The following statements are equivalent:
\begin{enumerate}
    \item The cotorsion pair $[ \mathsf{FP}_n\text{-}\mathsf{Proj}(R), \ \mathsf{FP}_n\text{-}\mathsf{Inj}(R) ]$
    is projective.
    \item Every projective \( R \)-module is \( \mathsf{FP}_n \)-injective, and the class 
    \(\mathsf{FP}_n\text{-}\mathsf{Inj}(R)\) satisfies the 2-out-of-3 property in \( \lMod R \).
    \item The ring \( R \) is left self-\( \mathsf{FP}_n \)-injective and left \( n \)-coherent, and the class 
    \(\mathsf{FP}_n\text{-}\mathsf{Inj}(R)\) is closed under kernels of epimorphisms. 
\end{enumerate}
\end{prop}

\begin{proof}
The implications \((1)\Leftrightarrow (2)\Rightarrow (3)\) are clear.\\
\((3)\Rightarrow (2)\).  
If \( R \) is left self-\( \mathsf{FP}_n \)-injective, it follows from \cite[Proposition~2.4]{Zhu} that every projective \( R \)-module is \( \mathsf{FP}_n \)-injective.  On the other hand, if \( R \) is left \( n \)-coherent, then the class \(\mathsf{FP}_n\text{-}\mathsf{Inj}(R)\) is closed under cokernels of epimorphisms. Since this class is always closed under extensions, it follows that \(\mathsf{FP}_n\text{-}\mathsf{Inj}(R)\) satisfies the 2-out-of-3 property in \( \lMod R \). 
\end{proof}

\begin{lemma}\label{lem:proj}
Let \( R \) be a ring, \( n \geq 1 \), and \( M \) a finitely \( n \)-presented \( R \)-module. Then:
\begin{enumerate}
    \item \( M \) is projective if and only if \( \Ext_R^1(M,N) = 0 \) for every finitely \( (n-1) \)-presented \( R \)-module \( N \).
    \item If \( R \) is left \( n \)-coherent, then \( M \) is projective if and only if \( \Ext_R^1(M,N) = 0 \) for every finitely \( n \)-presented \( R \)-module \( N \).
\end{enumerate}
\end{lemma}
\begin{proof} $(1)$ The "only if" direction is straightforward. For the converse, consider a short exact sequence $ 0 \to K \to F \to M \to 0,$ where \( F \) is finitely generated free and \( K \) is finitely \( (n-1) \)-presented. By hypothesis, \( \Ext_R^1(M,K) = 0 \), so the sequence splits and \( M \) is projective.

$(2)$ If \( R \) is left \( n \)-coherent, the class of finitely \( n \)-presented modules is closed under taking kernels of epimorphisms. Therefore, the argument from (1) applies with \( K \) and \( N \) finitely \( n \)-presented, proving that \( M \) is projective.
\end{proof}

\begin{theorem}
    Let \( R \) be a ring and $n\geq 1$. The following conditions are equivalent:
    \begin{enumerate}
        \item \( R \) is left \( n \)-von Neumann regular.
        \item $\mathsf{FP}_n(R)\subseteq {}^{\perp{_1}}\mathsf{FP}_{n-1}(R)$. 
        \item $R$ is left $n$-coherent and $\mathsf{FP}_n(R)\subseteq {}^{\perp{_1}}\mathsf{FP}_n(R)$. 
        \item $R$ is left $n$-coherent regular and every finitely generated projective submodule of a free $R$ module $F$ is a direct summand of $F$.
    \end{enumerate}
\end{theorem}
\begin{proof}
The equivalences \((1) \Leftrightarrow (2)\) and \((1) \Leftrightarrow (3)\) follow from Lemma \ref{lem:proj}.\\
($1 \Rightarrow 4$): Assume that \( R \) is left \( n \)-von Neumann regular. Then \( R \) is, in particular, left \( n \)-coherent regular. Let \( P \subseteq F \) be a finitely generated projective submodule of a free \( R \)-module \( F \). If \( F \) is finitely generated, the short exact sequence $ 0 \to P \to F \to F/P \to 0$
implies that \( F/P \) is finitely \( n \)-presented, hence projective by assumption. Therefore, \( P \) is a direct summand of \( F \). If \( F \) is not finitely generated, write \( F = F_1 \oplus F_2 \) with \( F_1 \) finitely generated free containing \( P \). By the previous argument, \( P \) is a direct summand of \( F_1 \), so $F = P \oplus P_0 \oplus F_2$ for some module \( P_0 \).\\
($4 \Rightarrow 1$):  Let \( M \) be a finitely \( n \)-presented \( R \)-module. Since \( R \) is left \( n \)-coherent regular, there exists a finite projective resolution of \( M \) of the form
$$0 \to P_k \xrightarrow{d_k} F_{k-1} \xrightarrow{d_{k-1}} \cdots \xrightarrow{d_2} F_1 \xrightarrow{d_1} F_0 \xrightarrow{d_0} M \to 0,$$
where each \( F_i \) is finitely generated free and \( P_n \) is finitely generated projective.

Decompose this resolution into short exact sequences:
$$0 \to P_k \xrightarrow{d_k} F_{k-1} \to \operatorname{Im} d_{k-1} \to 0,$$
$$0 \to \operatorname{Im} d_i \to F_{i-1} \xrightarrow{d_{i-1}} \operatorname{Im} d_{i-1} \to 0 \quad \text{for } 1 \leq i \leq k-1,$$
and
$$0 \to \operatorname{Im} d_0 \to F_0 \xrightarrow{d_0} M \to 0.$$

Since every finitely generated projective submodule of a free module is a direct summand, each sequence splits. Hence, by induction, \( M \) is projective.\\
\end{proof}

\begin{rmk} Let \( n \ge 1 \) and let \( I \) be any injective cogenerator in \( \lMod R\) (for example, \( I= \Hom_{\mathbb{Z}}(R, \mathbb{Q}/\mathbb{Z}) \)). Then \( R \) is left \( n \)-von Neumann regular if and only if, for every finitely \( n \)-presented \( R \)-module \( M \), the module $\Hom_R(M, I)$ is injective. This equivalence follows from the fact that \( R \) is left \( n \)-von Neumann regular if and only if every finitely \( n \)-presented \( R \)-module is flat, and from the characterization of flat modules via injectivity with respect to injective cogenerators, as established in \cite[Proposition 3.1]{CKWZ}.
\end{rmk}

Recall that an element \( m \) of an \( R \)-module \( M \) is called a \newterm{singular element} if the left ideal \( \mathsf{Ann}_R(m) \) is essential in \( R \), viewed as a left \( R \)-module. The set of all singular elements of \( M \) is denoted by \( \mathsf{Z}_S(M) \). We say that \( M \) is a \newterm{singular module} if \( \mathsf{Z}_S(M) = M \).

\begin{theorem}
    Let \( R \) be a ring and $n\geq 1$. Then \( R \) is left \( n \)-von Neumann regular if and only if every nonzero cyclic singular right \( R \)-module contains a nonzero \( \mathsf{FP}_n \)-flat submodule.
\end{theorem}

\begin{proof}
    Suppose that every nonzero cyclic singular right \( R \)-module has a nonzero \( \mathsf{FP}_n \)-flat submodule. Then, by \cite[Lemma 3.1]{FD10}, we have that \( \mathsf{FP}_n\text{-}\mathsf{Flat}(R^{\mathrm{op}})^{\perp_1} \subseteq \mathsf{Inj}(R^{op}) \). Consequently, by \cite[Theorem 3.7]{Zhu17}, \( R \) is left \( n \)-von Neumann regular. For the converse, note that if \( R \) is left \( n \)-von Neumann regular, then every right \( R \)-module is \( \mathsf{FP}_n \)-flat.
\end{proof}
    
\subsection{Commutative case} 
Following \cite[Theorem 2.1]{Mahdou},  a commutative ring \( R \) is \( n \)-von Neumann regular if and only if the following conditions hold: 
\begin{itemize}
    \item[a)] \( R \) is \( n \)-coherent,
    \item[b)] \( R \) is \( n \)-regular,
   \item[c)] Every finitely generated proper ideal of \( R \) has a non-zero annihilator.
\end{itemize}

 Condition (c)—that every finitely generated proper ideal of \( R \) has a non-zero annihilator—has been extensively studied by various authors \cite{Glaz1, HKLR, WZKXS}. For example, one fundamental property of commutative Noetherian rings is that the annihilator of an ideal \( I \), consisting entirely of zero-divisors, is always nonzero. However, this property does not generally hold for some non-Noetherian rings, even when \( I \) is finitely generated.  Many commutative rings satisfy condition (c). The following list, adapted from \cite{HKLR}, highlights some important examples:
\begin{itemize}
    \item Noetherian rings,  
    \item Rings whose prime ideals are maximal,  
    \item Polynomial rings of the form \( R[x] \), and  
    \item Rings whose classical rings of quotients are von Neumann regular.  
\end{itemize}

Condition (c) is equivalent to the property that every finitely generated projective submodule of a projective \(R\)-module \(P\) is a direct summand of \(P\); see \cite[Lemma 2.2]{Mahdou}. In particular, for a commutative local regular ring \((R, \mathsf{m})\), we have the following result:

\begin{prop} Let $(R,\mathsf{m})$ be a commutative local regular ring. Then \( \mathsf{FP}_{\infty}(R) = \mathsf{proj}(R) \) if and only if  $R$ satisfies condition (c).
\end{prop} 
\begin{proof}
    It is straightforward.
\end{proof}
\qed

If \( R \) is an \( n \)-von Neumann regular ring, every finitely \( n \)-presented ideal is projective. In the local case, such a ring has no nonzero proper finitely \( (n-1) \)-presented ideals (and hence no finitely \( n \)-presented ones); see \cite[Theorem 3.4]{Mahdou10}. The following result characterizes such rings.

\begin{prop}
    Let $(R,\mathsf{m})$ be a  commutative local ring and $n\geq 2.$ The following conditiosn are equivalent
    \begin{enumerate}
        \item $R$ is $n$-von Neumann regular.
        \item $\Ext_R^1(M, R/\mathsf{m})=0$ for every $M\in \mathsf{FP}_n(R)$.
        \item $\Tor^R_1(M, R/\mathsf{m})=0$ for every $M\in \mathsf{FP}_n(R)$.
    \end{enumerate}
\end{prop}

\begin{proof}
    \((1) \Leftrightarrow (3)\)  By \cite[Lemma 2.5.8]{Glaz1}, $M\in \mathsf{FP}_n(R)$ is projective if and only if $\Tor^R_1(M, R/\mathsf{m})=0$. \\
    \((2) \Leftrightarrow (3)\): By \cite[Theorem 3.4(2)]{BP}, for every finitely \( n \)-presented \( R \)-module \( M \), there is a natural isomorphism:
    \[
        \Tor_1^R(M, R/\mathfrak{m}) 
        \cong \Tor_1^R\left(M, \Hom_R(R/\mathfrak{m}, E(R/\mathfrak{m}))\right)
        \cong \Hom_R\left(\Ext_R^1(M, R/\mathfrak{m}), E(R/\mathfrak{m})\right),
    \]
    where \( E(R/\mathfrak{m}) \) denotes the injective envelope of \( R/\mathfrak{m} \).  Therefore, \( \Tor_1^R(M, R/\mathfrak{m}) = 0 \) if and only if \( \Ext_R^1(M, R/\mathfrak{m}) = 0 \).
\end{proof}

A ring \( R \) is \( n \)-von Neumann regular if and only if every  \( R \)-module is \( \mathsf{FP}_n \)-flat, and this is equivalent to \( \gfd(R) = 0 \). For \( n \)-coherent regular rings, we have the following characterization as stated in \cite[Corollary 4.10]{WZKXS}:

\begin{prop}
Let \( R \) be a commutative \( n \)-coherent regular ring with \( n \geq 1 \). The following are equivalent:
\begin{enumerate}
    \item \( R \) is  \( n \)-von Neumann regular.
    \item Every finitely generated \( R \)-module  is \( \mathsf{FP}_n \)-flat.
    \item Every cyclic \( R \)-module is \( \mathsf{FP}_n \)-flat.
    \item Every cyclic finitely presented \( R \)-module  is \( \mathsf{FP}_n \)-flat.
    \item \( \sfindim (R) = 0 \).
\end{enumerate}
\end{prop}
\qed

\begin{rmk}\label{rmk: dim weak}
    Let \( R \) be a commutative ring. If \( \wD(R) \leq n \), then \( R \) is \((n+1)\)-coherent and uniformly \((n+1)\)-regular, with the inclusion \( \mathsf{FP}_{n+1}(R) \subseteq \mathcal{P}_{n}(R) \); see Proposition \ref{prop: debil unif}. Moreover, if \( R \) is an integral domain, then \( R \) is \( n \)-coherent and uniformly \( n \)-regular, and we have \( \mathsf{FP}_{n}(R) \subseteq \mathcal{P}_{n}(R) \); see \cite[Theorem 4.5]{Costa}. Furthermore, by \cite[Theorem 4.1]{Costa}, if $\fd_R(\mathsf{m}) \leq n$ for every maximal ideal $ \mathsf{m}$ of $ R $, then \( R \) is $(n+2)$-coherent ring and uniformly \((n+2)\)-regular, with \( \mathsf{FP}_{n+2}(R) \subseteq \mathcal{P}_{n+1}(R) \). In particular, any commutative ring whose maximal ideals are flat is \( 2 \)-coherent and uniformly \( 2 \)-regular, and satisfies \( \mathsf{FP}_{2}(R) \subseteq \mathcal{P}_{1}(R) \). However, the converse does not hold: there exist commutative (even local) rings \( R \) such that \( \mathsf{FP}_{2}(R) \subseteq \mathcal{P}_{1}(R) \) but whose maximal ideal has infinite flat dimension; see \cite[Remark 3.5]{Mahdou}.
\end{rmk}

\section{Finiteness Conditions and $\mathsf{K}_0$-Regularity}\label{sec: aplicaciones}
In 1973, Quillen introduced the concept of $\mathsf{K}_i$-groups for exact categories. By considering the specific case of the category of finitely generated projective modules, one obtains the $\mathsf{K}_i$-groups of a ring. For any ring $R$, let $\mathbf{P}(R)$ denote the set of isomorphism classes of finitely generated projective $R$-modules. This set, equipped with the direct sum operation $\oplus$ and the identity element $0$, forms an abelian monoid. The Grothendieck group $\mathsf{K}_0(R)$ is defined as the quotient of the free abelian group generated by the isomorphism classes $[P]$ of finitely generated projective modules $P \in \mathsf{proj}(R)$, by the subgroup generated by elements of the form $[P \oplus Q] - [P] - [Q]$ for all $P, Q \in \mathsf{proj}(R)$. When $R$ is a commutative ring, the tensor product of two $R$-modules is again an $R$-module, with the property that $r \cdot (x \otimes y) = rx \otimes y = x \otimes ry$ for all $r \in R$, $x \in M$, and $y \in N$, where $M$ and $N$ are $R$-modules. Since $R^m \otimes_R R^n \cong R^{mn}$, it follows that $\mathsf{proj}(R)$ is closed under the tensor product $\otimes_R$. Consequently, the operation $[P][Q] = [P \otimes_R Q]$ induces a multiplication on $\mathsf{K}_0(R)$, making $\mathsf{K}_0(R)$ into a commutative ring, with identity element $[R]$.

For \( n \geq 0 \) and \( R \) a left \( n \)-coherent regular ring, the natural inclusion   $\mathsf{proj}(R) \hookrightarrow \mathsf{FP}_n(R)$
induces isomorphisms $\mathsf{K}_i(R) \cong \mathsf{K}_i(\mathsf{FP}_n(R)) \quad \text{for all } i \geq 0,$
see \cite[Theorem 3.2]{ep}. In the case where \( R \) is a left coherent regular ring, Swan \cite{Swan} proved that the natural map 
$
[M] \mapsto [R[t] \otimes_R M]$
induces an isomorphism 
$$\mathsf{K}_0(R) \cong \mathsf{K}_0(R[t]).$$
However, as Swan remarks in \cite[Remark 5.4]{Swan}, since \( R[t] \) need not be coherent even if \( R \) is, it remains unclear whether this result can be extended to \( R[t_1, \dots, t_n] \) for \( n > 1 \). 

A ring \( R \) is called \newterm{\( \mathsf{K}_0 \)-regular} if, for every integer \( n \geq 1 \), the natural homomorphism induced by the inclusion
$$\mathsf{K}_0(R) \longrightarrow \mathsf{K}_0(R[x_1, \dots, x_n])$$
is an isomorphism. In the commutative case, \( \mathsf{K}_0 \)-regularity holds under a regularity assumption:

\begin{prop}\label{prop:commutative-K0-regular}
\cite[Theorem 5.7]{Wang} Every commutative regular ring is \( \mathsf{K}_0 \)-regular.  
\end{prop}

As a consequence of Propositions \ref{prop:regularity-hierarchy-unified}, \ref{prop: equi regu}, and \ref{prop: sup}, together with Remark \ref{rmk: dim weak}, we obtain the following result.

\begin{cor}\label{ref:K0}
Let \( R \) be a commutative ring. If any of the following conditions holds, then \( R \) is \( \mathsf{K}_0 \)-regular:
\begin{enumerate}
    \item \( R \) is uniformly regular.
    \item \( R \) is uniformly \( n \)-regular for some \( n \geq 0 \).
    \item \( R \) is \( n \)-coherent regular for some \( n \geq 0 \).
    \item \( \mathsf{FP}_n(R) \subseteq \mathcal{P}_k(R) \) for some \( n, k \geq 0 \).
    \item \( \mathsf{FP}_n(R) \subseteq \mathcal{F}_k(R) \) for some \( n, k \geq 0 \).
    \item \( \mathsf{Max}(R) \subseteq \mathcal{F}_k(R) \) for some \( k \geq 0 \).
    \item \( R \) has finite weak dimension or finite global dimension.
    \item \( R \) is coherent and every cyclic finitely presented \( R \)-module has finite flat dimension.
\end{enumerate}
\end{cor}
\qed

\begin{rmk}\label{rmk: k-1}
    For coherent regular rings it is well known that \( \mathsf{K}_{-1}(R)=0 \). We now present an example of a \( 2 \)-coherent  regular ring, and hence \( K_0 \)-regular, for which  \( \mathsf{K}_{-1}(R) = \mathbb{Z} \). Let \( \mathsf{k} \) be an arbitrary field, and define
    $$A = \mathsf{k}[x_i \mid i \in [0, 1] \cap \mathbb{Z}[1/2]] \Big/ \left\langle (x_i - 1)x_j \mid i < j \right\rangle.$$
    Now let $R = A / \langle x_0 - x_1 \rangle,$ as in \cite[Example 5.20]{Aoki}. Since \( R \) has global dimension at most \( 2 \), it is a \( (0,2) \)-ring and, in particular, a \( 2 \)-coherent regular ring. Aoki further proves that \( \mathsf{K}_{-1}(R) = \mathbb{Z} \).
\end{rmk}

The \( \mathsf{K} \)-theory of arithmetical and valuation rings was studied in \cite[Propositions 6.1--6.2]{ep}. We now extend these results. For this purpose, recall that we write \( \mathsf{Z} \subseteq R \) for the set of zero divisors of the ring \( R \).

\begin{prop}
Let \( R \) be an arithmetical ring. If either of the following conditions holds, then \( R \) is \( \mathsf{K}_0 \)-regular:
\begin{enumerate}
    \item For every maximal ideal \( \mathsf{m} \) of \( R \) and every nonzero element \( x \in \mathsf{Z}(R_{\mathsf{m}}) \), the annihilator ideal $\mathsf{Ann}_R(x)$ is not finitely generated; 
    \item \( R \) is reduced.
\end{enumerate}
\end{prop}
\begin{proof}
In case (1), it follows from \cite[Proposition 3.8]{KMa2} that \( R \) is a \( (3,1) \)-ring.  In case (2), \( R \) is \( 2 \)-coherent by \cite[Theorem II.1]{Cou2}, and it is well known that reduced arithmetical rings are precisely those with \(\wD(R) \le 1\). Therefore, \( R \) is a \( (2,1) \)-ring.  In both cases, the claim then follows directly from Corollary \ref{ref:K0}.
\end{proof}

Recall that if \( A \subseteq B \subseteq \mathsf{Q}(A) \), where \( \mathsf{Q}(A) \) denotes the total ring of fractions of \( A \), and \( A \) is arithmetical, then \( B \) is also arithmetical. Using this observation, we obtain the following corollary:

\begin{cor}
Let \( R \) be a commutative reduced ring such that \( S \subseteq R \subseteq \mathsf{Q}(S) \) for some arithmetical ring \( S \). Then \( R \) is \( \mathsf{K}_0 \)-regular.
\end{cor}

\begin{prop}
Let \( R \) be a valuation ring such that, for every nonzero \( x \in \mathsf{Z}(R) \), the annihilator ideal \( \mathsf{Ann}_R(x) \) is not finitely generated. Then \( R \) is \( \mathsf{K}_0 \)-regular.
\end{prop}
\begin{proof}
Follows from \cite[Theorem 3.1]{KMa2} and Corollary \ref{ref:K0}.
\end{proof}

\subsection{Semi-valuation rings} For basic terminology, we refer to \cite{Dahl}. A \newterm{semi-valuation ring} is a ring \( R^+ \) equipped with a valuation map \( |{-}| \colon R^+ \to \Gamma \cup \{0\} \), where \( \Gamma \) is a totally ordered abelian group, satisfying the following conditions:
\begin{itemize}
    \item The set of zero divisors of \( R^+ \) is contained in the kernel \( \mathfrak{p} := \ker(|{-}|) \);
    \item For all \( x, y \in R^+ \) with \( |y| \neq 0 \) and \( |x| \leq |y| \), it holds that \( y \mid x \).
\end{itemize}
Associated to \( R^+ \), we define:
\begin{itemize}
    \item the \newterm{semifraction ring} \( R := R^+_{\mathfrak{p}} \),
    \item the \newterm{valuation ring} \( V := R^+ / \mathfrak{p} \),
    \item the residue field \( \mathsf{k} := R^+ / \mathfrak{p} \).
\end{itemize}

These give rise to the following commutative diagram:
$$\begin{array}{ccc}
R^+ & \longrightarrow & R \\
\downarrow & & \downarrow \\
V & \longrightarrow & \mathsf{k}
\end{array}$$
which is a \newterm{Milnor square}, that is, a bicartesian diagram with surjective parallel arrows; see \cite[Lemma 2.2]{Dahl}. There exist semi-valuation rings that are not coherent. For example, let \( R := \mathbb{Q}_p[[X,Y]] \) and let $
R^+ := \{ f \in R \mid f(0,0) \in \mathbb{Z}_p \},
$ be defined to be the pullback in the Milnor square:
$$\begin{array}{ccc}
R^+ & \longrightarrow & R \\
\downarrow & & \downarrow \text{ev}_{0,0} \\
\mathbb{Z}_p & \longrightarrow & \mathbb{Q}_p
\end{array}$$

More generally, if \( R^+ \) is a semi-valuation ring such that the localization \( R = R^+_{\mathfrak{p}} \) is not finitely generated as an \( R^+ \)-module, and \( \mathfrak{p} \) contains a regular sequence \( X, Y \), then \( R^+ \) fails to be coherent; see \cite[Lemma 3.11]{Dahl}.

\begin{prop}
Let \( R^+ \) be a semi-valuation ring with non-trivial valuation. Then \( R^+ \) is \(\mathsf{K}_0 \)-regular.
\end{prop}
\begin{proof}
By \cite[Lemma 3.7]{Dahl}, every finitely \( 2 \)-presented \( R^+ \)-module has projective dimension at most \( 1 \), hence \( R^+ \) is a \( (2,1) \)-ring. Thus, \( R^+ \) is \( \mathsf{K}_0 \)-regular by Corollary \ref{ref:K0}.
\end{proof}

\subsection{Completed group algebras} For the basic terminology used here, we refer to \cite{Burns}, \cite{Burns2}. Let $R$ be a commutative ring and $G$ a profinite group. The \newterm{completed group algebra} is defined as the inverse limit 
$$ R[[G]] := \varprojlim_{U} R[G/U], $$
where $U$ runs over the open normal subgroups of $G$, and the transition map for $U \subseteq U'$ is the group ring homomorphism $R[G/U] \to R[G/U']$ induced by the natural projection $G/U \to G/U'$.  As noted in \cite{Burns}, these algebras arise naturally in arithmetic contexts. For example, when \( R = \mathbb{Z} \), the algebra \( \mathbb{Z}[[G]] \) acts on inverse limits of modules, such as class groups and Selmer groups, over a tower of fields within a given Galois extension of number fields of group $G$.
 In particular, if $G$ has a countable basis of neighborhoods of the identity and a non-torsion Sylow subgroup, then $\mathbb{Z}[[G]]$ is neither left nor right coherent, as shown in \cite[Theorem 1.1]{Burns}.  More precisely, fixing a prime \( p \), \cite[Corollary 2.4]{Burns} shows that \( \mathbb{Z}[[G]] \) is neither left nor right coherent in each of the following cases:
\begin{itemize}
    \item[(i)] $G$ is a compact $p$-adic analytic group of positive rank.
    \item[(ii)] $G$ is the Galois group of an extension of number fields, or of $p$-adic fields, that contains a $\mathbb{Z}_p$-subextension for any prime $p$.
    \item[(iii)] $G$ is a Sylow $p$-subgroup of the absolute Galois group of a number field.
\end{itemize}

Furthermore, when \( R = \mathbb{Z} \) and \( G \) is a compact \( p \)-adic analytic group of rank \( d \), it follows from \cite[Theorem 1.2]{Burns} that \( \mathbb{Z}[[G]] \) is \( (d+3) \)-coherent. In the particular case \( G = \mathbb{Z}_p \), it follows from \cite[Theorem~1.1]{Burns2} that the completed group ring  
$$\mathbb{Z}[[\mathbb{Z}_p]] \cong \varprojlim_{n} \, \mathbb{Z}[\mathbb{Z}/p^n],$$ 
where the transition maps 
\(\mathbb{Z}[\mathbb{Z}/p^{n+1}] \to \mathbb{Z}[\mathbb{Z}/p^n]\) 
are induced by the canonical projections 
\(\mathbb{Z}/p^{n+1} \twoheadrightarrow \mathbb{Z}/p^n\), 
is not a finite conductor domain and hence not coherent. Moreover, if \( p \) is non-exceptional, then \(\mathbb{Z}[[\mathbb{Z}_p]]\) is a \( (2,2) \)-ring. Consequently, \(\mathbb{Z}[[\mathbb{Z}_p]]\) is uniformly \( 2 \)-regular, and therefore \( \mathsf{K}_0 \)-regular.

\begin{prop}
    For any non-exceptional rational prime \( p \), the completed group ring \( \mathbb{Z}[[\mathbb{Z}_p]] \) is \( \mathsf{K}_0 \)-regular.  
\end{prop}
\qed

\subsection{Group Algebras}  
For basic terminology, we refer to \cite{Kro1} and \cite{CK}. Let \( G \) be a group in the class \( \mathsf{H}\mathfrak{F} \) of hierarchically decomposable groups, defined in \cite{Kro1} as the smallest class containing all finite groups and closed under the following property: if \( G \) admits a finite-dimensional contractible \( G \)-CW-complex with stabilizers in \( \mathsf{H}\mathfrak{F} \), then \( G \) itself belongs to \( \mathsf{H}\mathfrak{F} \).  Moreover, if \( G \) is torsion-free (including torsion-free linear and soluble-by-finite groups), then every \( \mathbb{Z}G \)-module of type $\mathsf{FP}_{\infty}$ has finite projective dimension; see \cite[Corollary]{Kro1}. As a consequence, we obtain the following result:

\begin{prop}
    If $G$ is a torsion-free \( \mathsf{H}\mathfrak{F} \)-group then $\mathbb{Z}G$ is $\mathsf{K}_0$-regular. 
\end{prop}
\qed

\subsection{Perfect closure} For basic terminology, we refer to \cite{Glaz1} and \cite{Asgh}. Let \( R \) be a commutative ring of prime characteristic \( p \), that is $R$ contains a field of characteristic $p$, and let \( F : R \to R \) denote the \newterm{Frobenius map}, the ring homomorphism defined by \( F(x) = x^p \). Following \cite{Glaz1}, the \newterm{perfect closure} of \( R \), denoted \( R^\infty \), is given by the direct limit
$$R^\infty := \varinjlim \left( R \xrightarrow{F} R \xrightarrow{F} \cdots \right).$$
Note that \( R^\infty \) is always reduced. However, unless \( R \) is a field, the perfect ring \( R^\infty \) is almost never Noetherian. For example, if \( R \) is a Noetherian local ring of prime characteristic \( p \), then \( R^\infty \) is Noetherian if and only if \( \dim R = 0 \); see \cite[Observation 3.5]{Asgh}. If \(R\) is a one-dimensional complete local domain of prime characteristic, then by \cite[Corollary 3.14]{Asgh} we have \(\gD(R^\infty)=2\) if and only if \(R^\infty\) is stably coherent. When \(R^\infty\) is not coherent, it still holds that \(\gD(R^\infty)=3\), which implies that \(R^\infty\) is \( \mathsf{K}_0 \)-regular. In fact, by \cite[Fact 3.8]{Asgh} we obtain the following proposition. 
\begin{prop}
Let $(R, \mathsf{m})$ be a complete local ring of prime characteristic. Then its perfect closure \( R^{\infty} \) is \( \mathsf{K}_0 \)-regular.
\end{prop}
\qed

\subsection{ Ring of continous functions}
Let $\mathsf{C(X)}=\mathsf{C(X,\mathbb{R})}$ denote the ring of all continuous real-valued functions defined on a completely regular Hausdorff space $\mathsf{X}$, i.e., a Tychonoff space\footnote{Given \( f \in \mathsf{C}(X) \), the \newterm{zero set} of \( f \), denoted \( \operatorname{zer}(f) \), is the closed set 
$\operatorname{zer}(f) = \{x \in \mathsf{X} \mid f(x) = 0\},$
while the \newterm{cozero set} of \( f \), \( \operatorname{coz}(f) \), is the open set 
$\operatorname{coz}(f) = \{x \in \mathsf{X} \mid f(x) \neq 0\}.$
Since \( \mathsf{X} \) is Tychonoff, cozero sets form a base for the open topology on \( \mathsf{X} \).}. The set $\mathsf{C(X)}$ has a ring structure:  where both the sum and product of functions are defined pointwise, and the constant function equal to $1$ acts as the unity element.  Following \cite{GH}, a compactum $\mathsf{X}$ is called an \newterm{$\mathsf{F}$-space} if $\mathsf{C(X)}$ is Bézout ring. Every extremely disconnected compactum is an $\mathsf{F}$-space.  More generally, any basically disconnected compactum (i.e., a compactum in which the closure of every cozero set is open) is also an $\mathsf{F}$-space, as shown in \cite[Example 5.5]{Aoki}. For example, \( \mathsf{C([0,1])} \) is not an $\mathsf{F}$-space; see \cite[Remark 9.9]{MPR}.

\begin{prop}
Let $\mathsf{X}$ be an $\mathsf{F}$-space. Then $\mathsf{C(X)}$ is a $2$-coherent regular ring, and consequently, $\mathsf{C(X)}$ is $\mathsf{K}_0$-regular.
\end{prop}
\begin{proof}
Since $\mathsf{C(X)}$ is a reduced arithmetical ring, we have $\wD(\mathsf{C(X)}) \leq 1$. Moreover, by \cite[Theorem II.1]{Cou2}, $\mathsf{C(X)}$ is $2$-coherent. Therefore, $\mathsf{C(X)}$ is a $2$-coherent regular ring, and $\mathsf{K}_0$-regular by Corollary \ref{ref:K0}.
\end{proof}

A completely regular space \( \mathsf{X} \) is called a \newterm{$\mathsf{T}$-space} if the ring \( \mathsf{C(X)} \) is a Hermite ring. Consequently, every $\mathsf{T}$-space is an $\mathsf{F}$-space. However, it is known that there exist $\mathsf{F}$-spaces that are not $\mathsf{T}$-spaces; see \cite[Example 3.4]{GH}.  

\begin{prop}
If $\mathsf{X}$ is basically disconnected, then $\mathsf{X}$ is a $\mathsf{T}$-space. In particular, $\mathsf{C(X)}$ is $\mathsf{K}_0$-regular.
\end{prop}
\begin{proof}
If $\mathsf{X}$ is basically disconnected, then $\mathsf{C(X)}$ is semihereditary \cite[Fact 9.16]{MPR}, hence a Bézout ring \cite[Lemma 5.12]{AGG}, and therefore Hermite \cite[Corollary 3.2]{Cou1}. This shows that $\mathsf{X}$ is a $\mathsf{T}$-space.
\end{proof}

\section*{Acknowledgements}

The author would like to thank the anonymous referee for the careful reading of the manuscript and for the valuable comments and suggestions, which significantly improved the presentation of the paper.

The author was partially supported by the Agencia Nacional de Investigación e Innovación (ANII) and the Programa de Desarrollo de las Ciencias Básicas (PEDECIBA).

\end{document}